%% file: main.tex
\documentclass[11pt,a4paper]{article}
\usepackage{graphicx}
\usepackage{mathrsfs}
\usepackage{color}
\usepackage[utf8]{inputenc}
\usepackage{tikz-cd}

\usepackage{amsmath,amssymb, amsthm, latexsym}
\usepackage[colorlinks=true,linkcolor=magenta,citecolor=blue]{hyperref}
\input xy
\xyoption{all}

\author{Sébastien Alvarez and Pablo Lessa}
\date{}
\title{The Teichmüller space of the Hirsch foliation}

\begin{document}

\newtheorem{maintheorem}{Theorem}
\newtheorem{maincoro}[maintheorem]{Corollary}
\renewcommand{\themaintheorem}{\Alph{maintheorem}}
\newcounter{theorem}[section]
\newtheorem{exemple}{\bf Exemple \rm}
\newtheorem{exercice}{\bf Exercice \rm}
\newtheorem{conjecture}[theorem]{\bf Conjecture}
\newtheorem{defi}[theorem]{\bf Definition}
\newtheorem{lemma}[theorem]{\bf Lemma}
\newtheorem{proposition}[theorem]{\bf Proposition}
\newtheorem{coro}[theorem]{\bf Corollary}
\newtheorem{theorem}[theorem]{\bf Theorem}
\theoremstyle{remark}
\newtheorem{remark}{\bf Remark}
\newtheorem{ques}[theorem]{\bf Question}
\newtheorem{propr}[theorem]{\bf Property}
\newtheorem{question}{\bf Question}
\def\bp{\noindent{\it Proof. }}
\def\ep{\noindent{\hfill $\fbox{\,}$}\medskip\newline}
\renewcommand{\theequation}{\arabic{section}.\arabic{equation}}
\renewcommand{\thetheorem}{\arabic{section}.\arabic{theorem}}
\newcommand{\eps}{\varepsilon}
\newcommand{\disp}[1]{\displaystyle{\mathstrut#1}}
\newcommand{\fra}[2]{\displaystyle\frac{\mathstrut#1}{\mathstrut#2}}
\newcommand{\dif}{{\rm Diff}}
\newcommand{\homeo}{{\rm Homeo}}
\newcommand{\Per}{{\rm Per}}
\newcommand{\Fix}{{\rm Fix}}
\newcommand{\A}{\mathcal A}
\newcommand{\Z}{\mathbb Z}
\newcommand{\Q}{\mathbb Q}
\newcommand{\R}{\mathbb R}
\newcommand{\C}{\mathbb C}
\newcommand{\N}{\mathbb N}
\newcommand{\T}{\mathbb T}
\newcommand{\U}{\mathbb U}
\newcommand{\D}{\mathbb D}
\newcommand{\PP}{\mathbb P}
\newcommand{\Sp}{\mathbb S}
\newcommand{\K}{\mathbb K}
\newcommand{\car}{\mathbf 1}
\newcommand{\gs}{\mathfrak s}
\newcommand{\h}{\mathfrak h}
\newcommand{\rr}{\mathfrak r}
\newcommand{\ii}{\mathbf i}
\newcommand{\fhi}{\varphi}
\newcommand{\ffhi}{\tilde{\varphi}}
\newcommand{\moins}{\setminus}
\newcommand{\ds}{\subset}
\newcommand{\W}{\mathcal W}
\newcommand{\fT}{\mathbf T}
\newcommand{\WW}{\widetilde{W}}
\newcommand{\F}{\mathcal F}
\newcommand{\G}{\mathcal G}
\newcommand{\CC}{\mathcal C}
\newcommand{\RR}{\mathcal R}
\newcommand{\DD}{\mathcal D}
\newcommand{\M}{\mathcal M}
\newcommand{\B}{\mathcal B}
\newcommand{\cS}{\mathcal S}
\newcommand{\HH}{\mathcal H}
\newcommand{\NN}{\mathcal N}
\newcommand{\Hyp}{\mathbb H}
\newcommand{\UU}{\mathcal U}
\newcommand{\Pp}{\mathcal P}
\newcommand{\QQ}{\mathcal Q}
\newcommand{\E}{\mathcal E}
\newcommand{\GG}{\Gamma}
\newcommand{\LL}{\mathcal L}
\newcommand{\KK}{\mathcal K}
\newcommand{\TT}{\mathcal T}
\newcommand{\X}{\mathcal X}
\newcommand{\Y}{\mathcal Y}
\newcommand{\ZZ}{\mathcal Z}
\newcommand{\bE}{\overline{E}}
\newcommand{\bF}{\overline{F}}
\newcommand{\wF}{\widetilde{F}}
\newcommand{\hF}{\widehat{F}}
\newcommand{\hcF}{\widehat{\mathcal F}}
\newcommand{\bW}{\overline{W}}
\newcommand{\bcW}{\overline{\mathcal W}}
\newcommand{\tL}{\widetilde{L}}
\newcommand{\diam}{{\rm diam}}
\newcommand{\diag}{{\rm diag}}
\newcommand{\Jac}{{\rm Jac}}
\newcommand{\Mob}{{\rm M\ddot{o}b}}
\newcommand{\Rot}{{\rm Rot}}
\newcommand{\Trans}{{\rm Trans}}
\newcommand{\Plong}{{\rm Plong}}
\newcommand{\Tr}{{\rm Tr}}
\newcommand{\Conv}{{\rm Conv}}
\newcommand{\Ext}{{\rm Ext}}
\newcommand{\Spec}{{\rm Sp}}
\newcommand{\diffid}{{\rm Diff}_{Id}}
\newcommand{\Isom}{{\rm Isom}\,}
\newcommand{\Supp}{{\rm Supp}\,}
\newcommand{\Grass}{{\rm Grass}}
\newcommand{\Hold}{{\rm H\ddot{o}ld}}
\newcommand{\Ad}{{\rm Ad}}
\newcommand{\ad}{{\rm ad}}
\newcommand{\e}{{\rm e}}
\newcommand{\s}{{\rm s}}
\newcommand{\pol}{{\rm pole}}
\newcommand{\Aut}{{\rm Aut}}
\newcommand{\End}{{\rm End}}
\newcommand{\Leb}{{\rm Leb}}
\newcommand{\Liouv}{{\rm Liouv}}
\newcommand{\Lip}{{\rm Lip}}
\newcommand{\Int}{{\rm Int}}
\newcommand{\cc}{{\rm cc}}
\newcommand{\grad}{{\rm grad}}
\newcommand{\proj}{{\rm proj}}
\newcommand{\mass}{{\rm mass}}
\newcommand{\dive}{{\rm div}}
\newcommand{\dist}{{\rm dist}}
\newcommand{\im}{{\rm Im}}
\newcommand{\re}{{\rm Re}}
\newcommand{\codim}{{\rm codim}}
\newcommand{\Map}{\longmapsto}
\newcommand{\vide}{\emptyset}
\newcommand{\tr}{\pitchfork}
\newcommand{\ssl}{\mathfrak{sl}}
\renewcommand{\d}{\mathrm{d}}
\renewcommand{\l}{<}
\newcommand{\g}{>}
\newcommand{\mg}{\mathfrak g}

\newenvironment{demo}{\noindent{\textbf{Proof.}}}{\quad \hfill $\square$}
\newenvironment{pdemo}{\noindent{\textbf{Proof of the proposition.}}}{\quad \hfill $\square$}
\newenvironment{IDdemo}{\noindent{\textbf{Idea of proof.}}}{\quad \hfill $\square$}

\def\to{\mathop{\rightarrow}}
\def\act{\mathop{\curvearrowright}}
\def\To{\mathop{\longrightarrow}}
\def\Sup{\mathop{\rm Sup}}
\def\Max{\mathop{\rm Max}}
\def\Inf{\mathop{\rm Inf}}
\def\Min{\mathop{\rm Min}}
\def\lims{\mathop{\overline{\rm lim}}}
\def\limi{\mathop{\underline{\rm lim}}}
\def\egal{\mathop{=}}
\def\dans{\mathop{\subset}}
\def\surj{\mathop{\twoheadrightarrow}}

\def\h#1#2#3{ {\rm{hol}}^{#1}_{#2\rightarrow#3}}

\maketitle

\begin{abstract}
We prove that the Teichmüller space of the Hirsch foliation (a minimal foliation of a closed 3-manifold by non-compact hyperbolic surfaces) is homeomorphic to the space of closed curves in the plane.  This allows us to show that that the space of hyperbolic metrics on the foliation is a trivial principal fiber bundle.  And that the structure group of this bundle, the arc-connected component of the identity in the group of homeomorphisms which are smooth on each leaf and vary continuously in the smooth topology in the transverse direction of the foliation, is contractible.
\end{abstract}

\let\thefootnote\relax

\tableofcontents

\section*{Introduction}
\input{introduction.tex}

\section{Preliminaries}\label{sectionpreliminaries}
\input{preliminary.tex}

\section{Hyperbolic metrics on planar pairs of pants}\label{sectionpants}
\input{pants.tex}

\section{Model hyperbolic metrics on the Hirsch foliation}\label{sectionmodels}
\input{models.tex}

\section{Deforming hyperbolic metrics to model metrics}\label{sectiondeforming}
\input{deforming.tex}

\section{Contractibility of the space of foliation diffeomorphisms}\label{sectioncontractibility}
\input{contractibility}

\section*{Acknowledgements}

The authors were supported by a post-doctoral grant financed by CAPES respectively at IMPA in Rio de Janeiro and at UFRGS in Porto Alegre.  The second author was partially supported by CNPq project number 407129/2013-8 and SNI (Uruguay).

This paper grew out of visits of S.A. to IMERL in Montevideo, and of P.L. to IMB in Dijon and to IMPA in Rio de Janeiro. These visits were respectively financed by the IFUM, the ANR \emph{DynNonHyp} BLAN08-2 313375 and by IMPA. We wish to thank the warm hospitality and support of these institutions.

It is also a pleasure to thank Christian Bonatti and Bin Yu for discussions about $3$-dimensional topology, as well as Graham Smith and Alberto Verjovsky for discussions about the curve shortening flow, and José M. Rodriguez García for discussions about analytic capacity. Fernando Alcalde, Françoise Dal'bo, Matilde Martínez and Alberto Verjovsky shared their recent work on surface laminations, we are very grateful to all of them. A special thank goes to Michele Triestino who not only read and commented an early draft of this paper, but also brought our attention on Schwartz's work \cite{Schwartz1992}.  Finally, we also benefited from conversations with Bertrand Deroin, Matilde Mart\'inez, Rafael Potrie, and François Ledrappier on preliminary versions of this text.

This work would not have been possible without the encouragement of Sixto Rodriguez, and that's a concrete cold fact.

\input{main.bbl}

\begin{flushleft}
{\scshape S\'ebastien Alvarez}\\
Instituto Nacional de Matem\'atica Pura e Aplicada (IMPA)\\
Estrada Dona Castorina 110, Rio de Janeiro, 22460-320, Brasil\\
email: salvarez@impa.br

\smallskip

{\scshape Pablo Lessa}\\
IMERL, Facultad de Ingenier\'ia, Universidad de la Rep\'ublica\\
Julio Herrera y Reissig 565 CP11300.\\
Montevideo, Uruguay.\\
email: pablitolessa@gmail.com 
\end{flushleft}
\end{document}

%% file: introduction.tex
The Teichmüller space of compact surfaces is deeply related to the structure of the group of self-diffeomorphims of such surfaces.  For example, in \cite{EE} and \cite{ES}, it is shown that the identity component of the group of self-diffeomorphisms of a compact hyperbolic surface is contractible.  The proof proceeds in three steps.  First, one shows that the Teichmüller space is contractible.  From this one obtains that the space of hyperbolic metrics is a trivial principal fiber bundle over the Teichmüller space whose structure group is the identity component of the diffeomorphism group.  Finally, since the space of hyperbolic metrics is also contractible (as can be seen via identification with Beltrami coefficients), one obtains that the fiber must be as well.

In this work we extend the above line of reasoning to the Hirsch foliation which is a well known foliation of a closed 3-manifold by non-compact, non-simply connected, hyperbolic surfaces.   We prove the following results.

\begin{maintheorem}
\label{theorema}
The Teichmüller space $T(M,\F)$ of the Hirsch foliation is homeomorphic to the space $C(S^1,\R^{\ast}_+ \times \R)$ of continuous closed curves in the open half-plane.  
\end{maintheorem}

\begin{maintheorem}
\label{theoremb}
The space $H(M,\F)$ of hyperbolic metrics on the Hirsch foliation endowed with the projection to $T(M,\F)$ is a trivial principal $\diffid(M,\F)$ bundle.
\end{maintheorem}

\begin{maintheorem}
\label{theoremc}
The arc-connected component of the identity $\diffid(M,\F)$ in the group of homeomorphisms which are smooth on each leaf and vary continuously in the smooth topology in the transverse direction of the foliation, is contractible.
\end{maintheorem}

Theorem \ref{theorema} provides an infinite dimensional analogue of Fenchel-Nielsen coordinates for hyperbolic metrics for the Hirsch foliations. The two functions $\lambda:S^1 \to \R^{\ast}_+$ and $\tau:S^1 \to \R$ on the unit circle parametrizing each equivalence class of hyperbolic metrics can be interpreted as length and twist parameters associated to a certain family of disjoint closed geodesics on the leaves of the foliation.

We note that the Teichmüller theory of non-compact surfaces encounters several technical difficulties. In particular, for surfaces of infinite topological type, Teichmüller space can be defined, using either pants decompositions, complex structures, or length-spectrum.  But these definitions may yield different spaces as show in \cite{alessandrini-liu-papdopoulos2011,alessandrini-liu-papadopoulos2012} and \cite{alessandrini-liu-papadopoulos-su2012}.

Another way to generalize the concept of closed surface is to consider laminations of compact spaces by surfaces.  See \cite{ghys1999} for a discussion of the extent to which the basic theorems for Riemann surfaces, such as the Uniformization, Gauss-Bonnet, or the existence of meromorphic functions, can be extended to this general foliated context.

The problem of simultaneous uniformization of the leaves of a compact lamination is of special interest in the present work: is it possible to uniformize each leaf simultaneously? And, do the uniformizations of the leaves vary continuously with the transverse parameter? In his thesis \cite{Candel}, Candel answers both of these questions affirmatively in the case where the leaves are of hyperbolic type, by constructing families of hyperbolic metrics on the leaves which vary continuously with the transverse parameter.

Teichmüller spaces for hyperbolic surface laminations (where the total space of the lamination is compact even though the leaves may not be) were first introduced by Sullivan in \cite{Sullivan1988} and later on in more detail in \cite{Sullivan}.   They seem to be more amenable to study than the Teichmüller spaces of general non-compact surfaces.  For example, as in the case of compact surfaces one can define the Teichmüller space of a lamination either in terms of complex structures or hyperbolic metrics and both definitions are equivalent (see \cite[pg. 232]{MS}).  Related to this, Candel's version of the uniformization theorem (see \cite{Candel}) establishes that there is a unique hyperbolic structure conformal to each Riemannian structure on any lamination by hyperbolic leaves.

In spite of these results there are only two cases prior to this work in which one understands the Teichmüller space of a concrete surface lamination in any depth.

The first case is a certain family of laminations which can be associated to the expanding maps $z \mapsto z^d$ ($d > 1$) on the unit circle.  One can show that the Teichmüller space of such a lamination is in bijection with the $C^1$ conjugacy classes of expanding maps of the same degree, see \cite{Sullivan} and \cite{ghys1999}.

The second family of laminations are the so-called solenoids obtained as the inverse limit of the space of finite coverings of a closed hyperbolic surface (these are a special type of the solenoidal manifolds discussed in \cite{sullivan2014} and \cite{verjovsky2014}).   In this case, certain natural problems on finite coverings such as the Ehrenpreis conjecture, can be rephrased in terms of the Teichmüller space of the corresponding lamination (see \cite{penner-saric2008} and \cite{saric2009}).

Besides Candel's work there is, as far as the authors are aware, only one general result on Teichmüller spaces of hyperbolic surface laminations.  This is the fact, established by Deroin in \cite{deroin2007}, that the Teichmüller space of such a lamination containing a simply connected leaf is always infinite dimensional.

In this work we investigate the Teichmüller theory of the Hirsch foliation.  We recall that the Hirsch foliation was introduced in \cite{Hirsch1975} as an example of a foliation which is stable under $C^1$ perturbations of its tangent field, and has an exceptional minimal set (i.e. a compact closed set which is a union of more than one leaf and is not the entire ambient manifold).  We will study a variant (also considered by Ghys for example in \cite{ghys1995}) of his construction which is minimal (the difference with Hirsch's construction amounts to a different choice of degree $2$ map of the circle, we will use $z \mapsto z^2$ while he used a mapping with a single attracting periodic orbit).  The total space of the foliation can be defined as the orbit space of the wandering set of the domain of attraction of a solenoid mapping of the solid torus, and the foliation itself corresponds to the projection of the stable foliation of the attractor.  This actually describes a family of foliations which appear for example in the study of complex H\'enon maps (see \cite{HubOb1994}).  Our arguments apply to all of them and we will spend some time in the first section describing the different Hirsch foliations explicitly.

Hirsch foliations have also been used to produce examples of minimal foliations which are not uniquely ergodic (in the sense that they admit more than one harmonic measure, see \cite{deroin-vernicos2011}).

We believe that Hirsch foliations are of interest for the theory of Teichmüller spaces of hyperbolic surface laminations for the following reasons:  First, the previous laminations for which Teichmüller spaces are known all contain simply connected leaves and none of them contain complicated leaves (e.g. leaves whose fundamental group is infinitely generated).  Second, Sullivan solenoids are transversally Cantor and it is therefore easier to ``globalize'' local constructions on them.  Third, minimal surface laminations with an essential holonomy-free loop are Hirsch-like in a precise sense, e.g. all leaves are obtained by pasting together elements from the same finite set of compact surfaces with boundary (see \cite[Theorem 2]{alcalde-dalbo-martinez-verjovsky2014}).  Fourth, since the ambient space of the Hirsch foliation is a closed 3-manifold there may be deeper links between its Teichmüller theory and the dynamics and geometry of the manifold.

The solenoidal endomorphisms of $S^1 \times \C$ used to define the Hirsch foliation are known to be $C^1$ structurally stable when extended to $S^1 \times \widehat{\C}$ where $\widehat{\C}$ is the Riemann sphere, see \cite{iglesias-portela-rovella2010}.  Hence the $C^0$ conjugacy class of any such map contains a $C^1$ neighborhood of the map (see also \cite[Section 3]{HubOb1994} where the restrictions of the endomorphisms to a solid torus where they are injective are considered). In view of Sullivan's result relating $C^1$ conjugacy classes of expanding maps of the circle and the Teichmüller of a suitable lamination one might ask the following:

\begin{question}
What can be said about the $C^1$ conjugacy classes of solenoid mappings? In particular, is there a natural way of associating each such class to an element of the Teichmüller space of the associated Hirsch foliation?
\end{question}

The main technical issue which was solved in order to construct hyperbolic metrics on the Hirsch foliation was obtaining a global continuous section of the Teichmüller space of a pair of pants with several specific properties.  For proving that all metrics on the foliation are equivalent to a ``model metric'' we needed on the one hand a procedure for deforming metrics on the pair pants to metrics with the aforementioned special properties, and on the other hand a procedure for deforming a general metric on the Hirsch foliation so that it admits a specific set of closed curves as geodesics.   The most important tools we have used are geometric flows, in particular we use the flow on circle diffeomorphisms defined by Schwartz in \cite{Schwartz1992}, and the curve shortening flow on a hyperbolic surface (see \cite{Grayson1989}); as well as standard tools from Teichmüller theory such as the Beltrami equation (see for example \cite{ahlfors-bers1960}).

The Hirsch foliation we work with are constructed from a solenoidal mapping of degree $2$. It seems possible to generalize our construction of Teichmüller space to solenoidal mapping of degree $d$ with the additional condition that they are unbraided. Following \cite{HubOb1994}, this means that it sends a solid torus inside itself as a $(d,1)$ (unknotted) torus. in other words, the fundamental domain is diffeomorphic to the suspension of a $d+1$-connected plane domain by a rotation of angle $2\pi/d$.

We now review the structure of this paper.

In Section \ref{sectionpreliminaries} we construct the Hirsch foliation we will be working on and its Teichmüller space.  We also classify the different non-equivalent Hirsch foliations which arise from this type of construction.  Finally, we introduce a technique for deforming a hyperbolic metric defined on a neighborhood of a circle in the plane using a smooth isotopy of the identity in order to make it conformal and rotationally symmetric.

In Section \ref{sectionpants} we construct a global continuous section of the Teichmüller space of a planar pair of pants with several special properties.  The first of which is that the metrics in the section are conformal and rotationally symmetric around each boundary component.  Also, if one exchanges the length parameters for two ``legs'' of the pair of pants then the corresponding metrics given by the section differ by the 180º rotation (in particular if the two lengths are equal then the rotation is an isometry for the metric).  These properties are important in order to construct metrics on the Hirsch foliation which glue together smoothly under the identifications defining the foliation.

In Section \ref{sectionmodels} we use the preceeding global section to construct a family of hyperbolic metrics on the Hirsch foliation parametrized a length parameter $\lambda:S^1 \to \R_+$ and a twist parameter $\mu:S^1 \to \R$.  We also show that no two such metrics are Teichmüller equivalent.

In Section \ref{sectiondeforming} we prove that any hyperbolic metric on the foliation can be deformed (using a leaf-preserving isotopy which is leafwise smooth) to one of the model metrics.  This is done in two cases.  In the first one we assume that the given metric already has a certain distinguished family of curves as geodesics so one can operate separately on each pair of pants.  We then show how to deform a general metric to this case using the curve-shortening flow.

This concludes the proof of Theorem \ref{theorema}.  In Section \ref{sectioncontractibility} we show how to obtain Theorems \ref{theoremb} and \ref{theoremc}.

%% file: preliminary.tex
\subsection{The Hirsch foliation\label{hirschfoliation}}

\paragraph{Smale's Solenoid.} Consider the smooth endomorphism $f$ of $S^1\times \C$ (where $S^1 = \lbrace z \in \C: |z|=1\rbrace$) defined by
\[f(e^{it},z) = \left(e^{i2t},\frac{1}{2}e^{it} + \frac{1}{4}z\right).\]

Let $\fT = S^1 \times \D$ (where $\D$ is the open unit disk).  The closed solid torus $\overline{\fT}$ is mapped diffeomorphically into $\fT$ by $f$.  In fact, $f$ restricted to $\overline{\fT}$ is the well known \emph{solenoid map} and the compact set $K_0 = \bigcap\limits_{n \ge 0}f^{n}(\overline{\fT})$ is a \emph{hyperbolic attractor} which is locally homeomorphic to $\R$ times a Cantor set.

\paragraph{The Hirsch foliation.} Let $K$ be the union of preimages of $K_0$.  The quotient $M = ((S^1 \times \C) \setminus K)/f$ (where two points $x,y$ are equivalent if they belong to the same complete orbit, i.e. if $f^n(x) = f^m(y)$ for some $n,m \ge 0$) is a compact boundaryless smooth manifold and the foliation of $(S^1 \times \C) \setminus K$ by leaves of the form $\lbrace e^{it}\rbrace \times \C \setminus K$ descends to $M$.   The resulting foliated compact boundaryless $3$-manifold (where both the manifold and the leaves of the foliation are smooth) is what we will call from now on the \emph{Hirsch foliation} $(M,\F)$.

\paragraph{Topology of the leaves.} One can verify that each complete $f$-orbit intersects the set $M_0=\overline{\fT} \setminus f(\fT)$ (which is a $3$-manifold with two boundary components which are $2$-dimensional tori) at exactly one interior point or at one point on each boundary component: $M_0$ is a fundamental domain of $f$.  Hence $M$ is obtained by pasting the two boundaries of $M_0$ using $f$.  The sets of the form $P_t = M_0 \cap \lbrace e^{it}\rbrace \times \C$ are pairs of pants, and the partition of $M_0$ into these pants when pasted using $f$ yields the Hirsch foliation.  With this description it is simple to see that the leaves of the Hirsch foliation are homeomorphic to either the two-dimensional torus minus a Cantor set or the two-dimensional sphere minus a Cantor set (depending on whether the leaf contains a pair of pants $P_t$ such that $e^{it}$ is periodic under iterated squaring or not).

\subsection{Hyperbolic metrics and Teichmüller space}

\paragraph{Hyperbolic metrics.} By a \emph{hyperbolic metric on the Hirsch foliation} we mean an assignment of a hyperbolic Riemannian metric to each leaf which varies continuously transversally in local charts with respect to the topology of local smooth convergence.  The space of such metrics, endowed with the topology of locally uniform smooth convergence, will be denoted by $H(M,\F)$.

An \emph{identity isotopy} of the Hirsch foliation (sometimes we will just say \emph{leaf isotopy}, or \emph{leaf-preserving isotopy}) is a continuous function $I:[0,1]\times M \to M$ such that $I(0,\cdot)$ is the identity, and $I(s,\cdot)$ is a self-diffeomorphism when restricted to any leaf.  Furthermore one demands that in local foliated charts $I$ varies continuously in the smooth topology with respect to the transverse parameter.

Two hyperbolic metrics $g,g' \in H(M,\F)$ are said to be \emph{equivalent} if there exists an identity isotopy $I$ such that when $I(1,\cdot)$ is restricted to any leaf it is an isometry between $g$ and $g'$.  In other words the metric $g'$ is the pushforward of $g$ with respect to $I(1,\cdot)$.

\paragraph{Teichmüller space.} The \emph{Teichmüller space} $T(M,\F)$ of the Hirsch foliation is by definition the space of equivalence classes of Riemannian metrics under leaf-preserving identity isotopies.

In our special case these definitions can be given much more explicitly.   Any hyperbolic metric on the Hirsch foliation can always be lifted to $(S^1 \times \C) \setminus K$ yielding a $2\pi$-periodic family of metrics $g_t$, where $g_t$ is defined on $(\lbrace e^{it}\rbrace \times \C) \setminus K$.

Defining $K_t \subset \C$ so that $(\lbrace e^{it}\rbrace \times \C) \setminus K = \lbrace e^{it}\rbrace \times (\C \setminus K_t)$ one may identify each $g_t$ with a metric defined on $\C \setminus K_t$.

The transverse continuity of the hyperbolic metric translates as follows.  If $z \in \C \setminus K_t$ and $s_n \to t$ then there exists a compact neighborhood $U$ of $z$ such that all metrics $g_{s_n}$ with $n$ large enough are defined in $U$ and can be written as $a_n \d x^2 + 2b_n \d x \d y + c_n \d y^2$ where the functions $a_n,b_n$ and $c_n$ (the coefficients of $g_{s_n}$) converge in the smooth (i.e. $C^\infty$) topology to the corresponding coefficients for $g_t$ on $U$.

The definition of convergence in the space of hyperbolic metrics can be similarly written in these terms.  A sequence of metrics $g^n$ converges to a metric $g$ if and only if taking $g^n_t$ and $g_t$ as their lifts there exists for each $z \notin K_t$ a closed interval $I$ containing $t$ in its interior and a compact neighborhood $U$ of $z$ such that all metrics $g^n_s$ are defined on $U$ for all $s \in I$ and converge to $g_s$ on $U$ in the smooth topology uniformly with respect to $s \in I$.

\subsection{General Hirsch foliations}
\label{canonicalhirsch}
The Hirsch foliation we described before is a very concrete algebraic model. Hirsch's original construction \cite{Hirsch1975} is more topological. We wish to prove here that the space of metrics we will describe below does not depend on the algebraic model we chose.

The sequel seems folklore and must be well known to the specialist. But even though the Hirsch foliation has been studied for some time now (see for example \cite{candel-conlon2000,deroin-vernicos2011,ghys1995}), it has been difficult to locate a careful construction of the different Hirsch foliations. Hence we found useful to give a topological discussion about it. For basic $3$-manifold theory we refer to \cite{Hatcher2007}.

\input{canonicalhirsch.tex}

\subsection{Massage of an annulus}
\label{massageannuli}

The goal of this paragraph is to describe a procedure for deforming hyperbolic metrics around a geodesic circle $C$ via identity isotopy.   This procedure can later be applied around each boundary component of $P$ (via the affine maps $z \mapsto z/4 \pm 1/2$ for the left and right boundaries) to construct the homotopy of Theorem \ref{admissiblehomotopy}.   However, we will also use the procedure directly later on for deforming hyperbolic metrics on the Hirsch foliation (see the proof of Lemma \ref{geodesicmeridians}).

\input{massage.tex}

%% file: canonicalhirsch.tex
\subsubsection{Seifert bundle over the pair of pants}

\paragraph{A suspended manifold.}
Here a pair of pants $P$ will be a surface with three boundary components which is diffeomorphic to the symmetric planar pair of pants $\{z\in\C:\,|z|\leq 1,|z\pm 1/2|\geq 1/4\}$.

Consider $\phi:P\to P$ an orientation preserving diffeomorphism which:
\begin{itemize}
\item leaves invariant one of the boundary components, which we call the \emph{outer component};
\item exchanges the other two boundary components, which we will call the \emph{inner components};
\item has a unique fixed point in $P$ denoted by $p_0$;
\item is of order two.
\end{itemize} 

In the symmetric case, just consider the rotation of angle $\pi$. Suspend this diffeomorphism to construct the following manifold:
$$M_0=P\times\R/\{(x,t)\sim (\phi(x),t+1)\}.$$

This manifold fibers over the circle, with a $P$-fiber. From now on we will refer to the boundary components of the $P$-fibers of $M_0$ as \emph{meridians}.

The manifold $M_0$ is a solid torus with an inner solid torus drilled out which winds around twice longitudinally while winding once meridianally. It has two boundary components which are tori, we will call the outer boundary torus $T_{out}$ and inner one $T_{inn}$ respectively. The discussion below is again valid with any equivalent pant bundle over the circle which is equivalent to $M_0$.

We recall that a surface $S$ embedded in a $3$-manifold $M_0$ is \emph{incompressible} if the morphism $\pi_1(S)\to\pi_1(M_0)$ induced by the inclusion is injective.

\begin{lemma}
\label{Seifertincompressible}
The boundary components of $M_0$ are incompressible.
\end{lemma}

\begin{proof}
The fundamental group of $M_0$ is given by a semi-direct product $\pi_1(M_0)=F_2\rtimes_{\phi_{\ast}}\Z$, where $F_2=\pi_1(P)$ denotes the free group with two generators (corresponding for example to the two inner boundary components of $P$), and $\phi_{\ast}$ is the morphism of $F_2$ induced by $\phi$: it permutes the two generators of $F_2$.

A meridian of $T_{out}$ represents a non trivial element of the $F_2$ factor (the product of the two generators). As for the $\Z$ factor, it can be represented by the longitud of $T$. Hence, the inclusion $T_{out}\hookrightarrow M_0$ induces an injection $\pi_1(T_{out})\to\pi_1(M_0)$: $T_{out}$ is incompressible. The same argument provides the incompressibility of $T_{inn}$.
\end{proof}

\paragraph{Structure of Seifert bundle.} The suspension flow on $M_0$ defines a structure of \emph{Seifert bundle} over $P$ with a unique exceptional fiber of type $(1,2)$ \cite{Hatcher2007}, which corresponds to the fixed point $x_0$ of $\phi$. We will denote this exceptional fiber by $S_0$.

Remember that a $3$-manifold is said to be \emph{irreducible} if every embedded $2$-sphere bounds a $3$-dimensional ball. By Proposition 1.12 of \cite{Hatcher2007}, the manifold $M_0$ is irreducible (it is clearly not one of the exceptions listed in this proposition).

\begin{lemma}
\label{uniqueSeifertfiber}
Every diffeomorphism of $M_0$ is homotopic to a Seifert fiber preserving diffeomorphism.
\end{lemma}

\begin{proof}
See the classification of Seifert bundles given in Theorem 2.3 of \cite{Hatcher2007}.
\end{proof}

The orientations of the pair of pants and of $S^1$ provide a natural orientation on $M_0$. With this orientation, the inner (resp. outer) Seifert fiber and the meridian provide two homology classes $\alpha_{inn}$ (resp. $\alpha_{out}$) and $\beta$ of $T_{inn}$ (resp. $T_{out}$), and they have intersection number $1$ (resp. $2$).

\begin{lemma}
\label{transversepants}
Let $\Pp$ be a foliation of $M_0$ by pants which are transversal to the Seifert fibration.  Then there exists and integer $d$ such that the outer boundary component of all pairs of pants in $\Pp$ are in the homology class $d\alpha_{out}+\beta_{out}$ on $T_{out}$ and the inner boundaries of all such pants are the class $d\alpha_{inn}+\beta_{inn}$ on $T_{inn}$.
\end{lemma}

\begin{proof}
The manifold $M_0$ is naturally a circle bundle over the $2$-orbifold $\Sigma_0$ obtained by quotienting $P$ by $\phi$, which is homeomorphic to an annulus. Consider an arc $c$ in $\Sigma_0$ linking the boundaries, and lift it to $M_0$. We obtain an annulus $A$ everywhere transverse to the fibration $(P_t)_{t\in S}$. The boundary components of this annulus are by definition outer and inner Seifert fibers.

Now, note that $\Sigma=\Sigma_0\moins c$ is simply connected, in such a way that the restriction of the fiber bundle to $M_0\moins A$ is a trivial circle fibration over the simply connected manifold with boundary, which, topologically, is a closed band. The boundary of $M_0\moins A$ is therefore the union of four annuli, two of which, denoted by $A_1$ and $A_2$, are copies of $A$.

In particular, since the base $\Sigma_0\moins c$ is simply connected, any foliation such as described in the statement of the lemma provides a foliation of $M_0\moins A$ which is isotopic along the fibers to the trivial product foliation. The gluing of $A_1$ and $A_2$ determines the type of foliations.

Now, such a gluing is determined by Dehn twists along $A$ (since any class of isotopy of diffeomorphisms of the annulus is represented by a Dehn twist). Since boundary components of $A$ are Seifert fibers, we can conclude the proof of the lemma.
\end{proof}

\subsubsection{Gluing the boundary components}

\paragraph{The Hirsch foliation.}

Consider an orientation preserving diffeomorphism $f:T_{out}\to T_{inn}$ which sends meridians onto meridians, and consider the manifold $M_f$ obtained by gluing the two boundary components of $M_0$ using $f$.

The manifold $M_0$ is foliated by pairs of pants, which induce two foliations of the boundary components by circle (the two meridian foliations). By definition, $f$ sends the first meridian foliation onto the second meridian foliation. The gluing by $f$ then provides a foliation of $M_f$ by surfaces that we denote by $\F$ and that we call the  \emph{Hirsch foliation} associated to $f$.

The circle $S_0$ yields naturally a circle $S$ in $M_f$ which is transverse to all leaves of $\F$. Hence the foliation is \emph{taut} in the sense of \cite{Calegari2007}.

\paragraph{The graph manifold.}  The manifold $M_f$ possesses a natural torus $T$, that as we prove later, is canonical in the sense that it is the unique \emph{JSJ torus} of $M_f$. The resulting manifold is called a \emph{graph manifold}, and is not a Seifert bundle itself.

\begin{lemma}
\label{notseifert}
Let $f:T_{out}\to T_{inn}$ be an orientation preserving diffeomorphism that preserves the meridians. Then $M_f$ is not a Seifert bundle.
\end{lemma}

\begin{proof}
We are going to work inside $M_0$. $M_0$ possesses a unique structure of Seifert manifold (see Lemma \ref{uniqueSeifertfiber}), so it is enough to see that $f$ does not send the Seifert fibers of $T_{out}$ onto that of $T_{inn}$.

A Seifert fiber of $T_{out}$ intersects each meridian twice, and a Seifert fiber of $T_{inn}$ intersects each meridian only once. Since $f$ sends diffeomorphically meridians of $T_{out}$ onto meridians of $T_{inn}$, it implies that $f$ cannot preserve the Seifert fibers. This concludes the proof.
\end{proof}

\begin{lemma}
\label{irreducibleincompressible}
Let $f:T_{out}\to T_{inn}$ be an orientation diffeomorphism that preserves the meridians. Then $M_f$ is irreducible and the torus $T$ is incompressible.
\end{lemma}

\begin{proof}
Let us prove first that $T$ is incompressible. We will use the Loop Theorem \cite{Hatcher2007}: let $D\dans M_f$ be an embedded closed $2$-disc such that $D\cap T=\partial D$ is an embedded circle in $T$. We have to prove that $\partial D$ is null homotopic in $T$. The interior of $D$ does not meet $T$: there is an embedded copy $D_0$ of $D\moins\partial D$ inside $M_0$ such that $\partial D_0$ is included in one of the two boundary components $T_{out}$ or $T_ {inn}$. It is enough to prove that $\partial D_0$ is null homotopic inside this component. But this is true since $M_0$ has an incompressible boundary.

Now let us prove that $M_f$ irreducible. Notice that $M_f$ has an incompressible torus: in particular its fundamental group possesses a copy of $\Z^2$ and is not finitely covered by $S^2\times S^1$. Moreover it possesses a taut foliation, so by Novikov's theorem (see Theorem 4.35 of \cite{Calegari2007}), it is irreducible.
\end{proof}

\paragraph{The JSJ torus.} We show now that inside the manifold $M_f$, the torus $T$ is canonical.

Recall that any compact and irreducible $3$-manifold can be canonically decomposed into pieces that are either Seifert or atoroidal (any incompressible torus is isotopic to a boundary component) and acylindrical (any properly embedded annulus is isotopic, fixing the boundary, to a subannulus of a boundary component) by cutting along a collection of incompressible tori. Such a collection of tori is unique up to isotopy. This decomposition is called the \emph{JSJ decomposition}, and decomposition tori are called \emph{JSJ tori} (see \cite{Hatcher2007}).

\begin{lemma}
$T$ is, up to isotopy, the only incompressible torus of $M_f$.
\end{lemma}

\begin{proof}
By Lemma \ref{irreducibleincompressible} $M_f$ is irreducible and $T$ is incompressible. On the other hand, $M_f$ is not a Seifert bundle (Lemma \ref{notseifert}), while $M_f\moins T$ is: hence, $T$ is the unique JSJ torus of $M$.
\end{proof}

\paragraph{Uniqueness of the Hirsch foliation.}
Now we intend to prove that the Hirsch foliation is unique. To see this, imagine that there is another fibration by pairs of pants of $M_0$, that we denote by $(P_t')_{t\in S_0}$ which is everywhere transverse to the Seifert bundle and is \emph{$f$-invariant} in the following sense. 

The boundaries of the pairs of pants $P_t'$ determines two foliations of $T_{out}$ and $T_{inn}$ that we call $P'$-meridians (usual meridians will also be called $P$-meridians) which are everywhere transverse to the Seifert fibers. We say the the family $(P_t')_{t\in S_0}$ is $f$-invariant if $f$ preserves the $P'$-meridians. Note that in that case the boundary components of $P'_t$ and that of $P_t$ are freely homotopic: this comes from Lemma \ref{transversepants} and from the fact that $f$ does not preserve the Seifert fibers. Then gluing $T_{out}$ and $T_{inn}$ by $f$ provides another foliation $\F'$ on $M_f$. The next lemma implies that the this new foliation is isotopic to $\F$.

\begin{lemma}[Uniqueness of the Hirsch foliation]
\label{uniquenesshirsch}
The Hirsch foliation of $M_f$ is unique. More precisely, consider another fibration $(P_t')_{t\in S_0}$ in pairs of pants transverse to the Seifert bundle, which is $f$-invariant. Then there exists an isotopy $\Phi_s:M_0\to M_0$ such that for every $t\in S_0$, $\Phi_1(P_t')=P_t$ and which commutes with $f$: $\Phi_s\circ f=f\circ\Phi_s$ in $T_{out}$.
\end{lemma}

\begin{proof}
Let $(P_t')_{t\in S_0}$ be a fibration of $M_0$ in pairs of pants everywhere transverse to the Seifert bundle, which is $f$-invariant. As we noted before, $P$ and $P'$-meridians are freely homotopic. Hence if one lifts the fibration to $P\times\R$, the $P'$-meridians lift as simple closed curves (they are freely homotopic to the lifts of $P$-meridians).

By hypothesis, all the pairs of pants $P_t'$ are everywhere transverse to the lines $\{x\}\times\R$: $P_t'$ may be see as a graph of a smooth function $\fhi_t:P\to\R$ satisfying the equivariance relation $\fhi_{t+1}=\fhi_t+1$.

In particular, the lifts of the outer $P'$-meridians are graphs over those of outer $P$-meridians. Using the vertical flow and the function $\fhi_t$ above, it is possible to isotope $P'$-meridians to corresponding $P$-meridians. Pushing this isotopy by $f$ shows how to isotope the inner $P'$-meridians onto inner $P$-meridians. These isotopies may be extended to neighborhoods of $T_{inn}$ and $T_{out}$ in order to isotope $(P_t')_{t\in S_0}$ to a family $(P_t'')_{t\in S_0}$ sharing the same properties, and coinciding moreover with $(P_t)_{t\in S_0}$ near the boundary, via an isotopy which commutes with $f$.

Using one more time that the function $\fhi_t$ in $P\times\R$ enables us to glue the isotopy above with an isotopy which sends the interior of $P_t''$ to that of $P_t$. The resulting isotopy stays $f$-invariant.
\end{proof}

\subsubsection{A homological invariant}

\paragraph{The twisting number.} Define $d_f$ as the intersection number, inside $T_{inn}$, of the homology classes $\alpha_{inn}$ and $f_{\ast}\alpha_{out}$.

\begin{lemma}
\label{odd}
The intersection number $d_f$ is always odd.
\end{lemma}

\begin{proof}
First, in a natural basis of integer homology of $T_{out}$, $\alpha_{out}$ can be written as $(1,2)$. In particular, it is not the power of some homology class.

Inside $T_{inn}$, we have a natural basis of the homology defined by $\alpha_{inn}$ and $\beta$ (their intersection number is $1$). Since $f$ sends meridian onto meridian (and preserves the orientation), and the representation of $f_{\ast}\alpha_{out}$ in this basis is $(d_f,2)$. Since this is not the power of some homology class (the action of $f_{\ast}$ on the homology is invertible), $d_f$ has to be odd.
\end{proof}

\paragraph{A topological invariant.}
\begin{theorem}
\label{topologicalinvariant}
Let $f,f':T_{out}\to T_{inn}$ be two meridian preserving diffeomorphisms. Then the following properties are equivalent.
\begin{enumerate}
\item $d_f=d_{f'}$.
\item $M_f$ and $M_{f'}$ are diffeomorphic.
\item There exists a diffeomorphism $H:M_0\to M_0$ which conjugates $f$ and $f'$.
\item The Hirsch foliations corersponding to $f$ and $f'$are conjugate.
\end{enumerate}
\end{theorem}

\begin{proof}
First, note that the third and fourth assertion clearly imply the second one.

Assume that $d_f=d_{f'}$. Then since $f$ and $f'$ send diffeomorphically meridian on meridian, we see that they induce the same action in the first homology of the $2$-torus. Hence they have the same isotopy type, and the glued manifolds $M_f$ and $M_{f'}$ are diffeomorphic.

Now, assume that $M_f$ and $M_{f'}$ are diffeomorphic: denote by $H$ a diffeomorphism between them. Since $H$ is a diffeomorphism, $H(T_f)$ is incompressible: it is isotopic to $T_{f'}$. After performing an isotopy, one can ask that $H(T_f)=T_{f'}$. 

This implies that $H$ induces a diffeomorphism of $M_0$, still denoted by $H$, such that the commutation relation $f'\circ H=H\circ f$ holds in restriction to $T_{out}$. By Lemma \ref{uniqueSeifertfiber} $H$ is homotopic to a fiber preserving diffeomorphism. Hence, $H$ preserves the homology classes $\alpha_{inn}$ and $\alpha_{out}$. Since it conjugates $f$ and $f'$, we deduce that $d_f=d_{f'}$.

We want to prove that in that case, the corresponding Hirsch foliations are conjugate. The image by $H$ of the fibration $(P_t){t\in S_0}$ provides a family of pairs of pants which is transverse to the Seifert bundle (since $H$ preserves it) and is $f'$-invariant (since $H$ conjugates the actions of $f$ and $f'$). Lemma \ref{uniquenesshirsch} provides an isotopy from $(P_t')_{t\in S_0}$ to $(P_t)_{t\in S_0}$ which is $f'$-invariant. In other terms, the two Hirsch foliations are conjugate.

The other implications are obvious.

\end{proof}

This homological invariant $d_f$ will be referred to as the \emph{twisting number} of the corresponding Hirsch foliation.

\subsubsection{Algebraic models}
\paragraph{The solenoid.} We have already met Smale's solenoid $f:S^1\times\C\to S^1\times\C$:

$$f(e^{i t},z)=\left(e^{2it},\frac{1}{2}e^{it}+\frac{1}{4}z\right).$$

Identify $M_0$ and a fundamental domain of $f$ given by $\overline{\fT}\moins f(\fT)$. Then, $f$ clearly induces a diffeomorphism (still denoted by $f$) from $T_{out}$ to $T_{inn}$ which preserves the meridians, and satisfies $d_f=1$.

Hence every Hirsch foliation with same twisting number $1$ is conjugated to this model, which we will study in detail in what follows.

\paragraph{Twisted model.} It is easily showed that if in $M_0$, we compose the diffeomorphism with a positive Dehn twist of $T_{inn}$, the twisting number is increased by $2$. Hence these diffeomorphisms provide models of the Hirsch foliation for every odd integer $d_f$. Algebraic models exist, and appear in \cite{HubOb1994}. They are defined by maps $f_k:S^1\times\C\to S^1\times\C$ given by the formula:

$$f_k(e^{i t},z)=\left(e^{2it},\frac{1}{2}e^{it}+\frac{e^{kit}}{4}z\right).$$

%% file: massage.tex
\subsubsection{Standard hyperbolic annuli}

\paragraph{Standard hyperbolic annuli.} For each positive length $\ell$ there is a unique conformal metric $\sigma_\ell$ on the annulus $\A_\ell = \lbrace z\in\C:\,e^{-\pi^2/\ell} \l |z| \l e^{\pi^2/\ell}\rbrace$ in $\C$ which is hyperbolic, rotationally invariant, and such that the unit circle $r = 1$ is a geodesic of length $\ell$.

In order to see this consider the strip model of the hyperbolic plane.  That is, consider the metric
\[\d s^2 = \frac{1}{\cos(y)^2}(\d x^2 + \d y^2)\]
on the strip $\lbrace x+i y \in \C: -\pi/2 \l y \l \pi/2\rbrace$ (this is obtained from the usual upper half plane model by pullback under the conformal map $z \mapsto \exp(-i z)$).

The metric $\sigma_\ell$ is the pushforward of the above metric via the conformal covering map $z \mapsto \exp(2\pi i z/\ell)$.

The couples $(\A_{\ell},\sigma_{\ell})$ will be referred to as the \emph{standard hyperbolic annuli}.

\paragraph{Standard hyperbolic metrics.}
Suppose that $g$ is a hyperbolic Riemannian metric defined on some region in the plane for which a Euclidean circle $C$ is a closed geodesic of length $\ell$.   We say $g$ is standard around $C$ if it coincides, on some neighborhood of $C$, with the pushforward of $\sigma_\ell$ under a conformal map of the form $z \mapsto az + b$ taking the unit circle to $C$.

\subsubsection{The Massage Lemma}

Consider an annulus $\A=\{z\in\C:\,0.9<|z|<1.1\}$ and denote by $H(\A)$ the set of hyperbolic metrics on $\A$ with the unit circle as a geodesic. Consider $R:\C\to\C$ the rotation of angle $\pi$, i.e. $R(z)=-z$.  We will prove the following:
\begin{theorem}[Massage Lemma]
\label{massagelemma}
For each $g\in H(\A)$ there exists an identity isotopy $F_{s,g}:\A \to \A$ such that:
\begin{enumerate}
\item Each diffeomorphism $F_{s,g}$ preserves the unit circle and is the identity outside the annulus defined by $\A'=\{z\in\C:0.91<|z|<1.09\}$.
\item The pullback metric $(F_{1,g})^{\ast}g$ is standard around $C$.
\item If $A(s,g)=(F_{s,g})^*g$, then we have for every $s\in[0,1]$, $R_{\ast}A(s,g)=A(s,R_* g)$.
\item For each $s\in[0,1]$ the map $g \mapsto F_{s,g}$ is continuous in the smooth topology.
\end{enumerate}
\end{theorem}

The deformation will follow four steps. At each step of this procedure, we shall check that Properties 1. 2. 3. 4. are satisfied.
\begin{enumerate}
 \item \emph{Deformation of the unit tangent bundle}. We perform an identity isotopy for the circle $C$ to be geodesic when parametrized by Euclidean arc length.  In other words, after this step the Euclidean unit tangent vector field to the circle is parallel for the metric.
 \item \emph{Deformation of the normal bundle}. We perform an identity isotopy so that the Euclidean normal vector field to $C$ is also perpendicular to $C$ for the deformed metric.
 \item \emph{Conformality on the circle}. We perform an isotopy in order to make the metric conformal to the Euclidean metric on the circle $C$.
 \item \emph{Standardness around the circle}. We perform an isotopy to make the metric standard around $C$.  At this point (and only at this point) some standard metrics may be deformed to other equivalent standard metrics.
\end{enumerate}

\subsubsection{Auxiliary functions}\label{auxiliarystep}

We fix from now on a bump function, i.e. a smooth map $\rho:\R \to \R$ which is $1$ on $[0.99,1.01]$ and $0$ on $(-\infty,0.91]\cup [1.09,+\infty)$. 

For each $\lambda \g 0$ we will also need to fix a smooth increasing diffeomorphism $f_\lambda:\R \to \R$ which is the identity outside of the interval $[0.9,1.1]$ and has derivative $\lambda$ at $1$.   We further suppose that $f_\lambda(x)$ is smooth with respect to both $\lambda \g 0$ and $x \in [0,1]$ and that that $f_1$ is the identity map.

\subsubsection{Deformation of the tangent bundle}

\paragraph{Parametrize the geodesic.} Until the end of the proof of the Massage Lemma (Theorem \ref{massagelemma}), we fix a metric $g\in H(\A)$. Let $\ell$ be the $g$-length of the geodesic $C$. In order to simplify the presentation, we will assume, in this paragraph only, that $\ell=2\pi$.

Consider an \emph{arc length parametrization} of $C$, that is an orientation preserving diffeomorphism $\gamma\in\dif^{\infty}_+(S^1)$ such that $g=\gamma_* \d\theta^2$ where by definition $\d\theta^2$ is the Euclidean metric on $C$. Such a diffeomorphism is well defined up to precomposition by a rotation.

There is a continuous section $g\mapsto\gamma_g$, which is characterized by the condition $\gamma_g(1)=1$.

We will denote by $\Mob(C)$, the group of Möbius transforms leaving $C$ invariant, by $\Rot$ the group of rotations of $C$, $z\mapsto e^{it} z$, and by $\Trans$, the group of translations of $\R/2\pi\Z$. These groups are conjugated by the exponential map $E:\R/2\pi\Z\to S^1$, $t\mapsto e^{i t}$.

\paragraph{A bi-equivariant deformation retract.} We want to find a continuous path of metric $g$ to a metric for which $C$ is a geodesic when parametrized by the Euclidean length of arc. And we want these paths to be equivariant for the action of $R$. A first step is to modify in an equivariant way the parametrization of $C$. In order to do this, we will use a strong deformation retract from $\dif^{\infty}_+(C)$ to $\Rot$, which is bi-equivariant for the action of $\Rot$. This requires a construction due to Schwartz \cite{Schwartz1992}.

\begin{theorem}[Bi-equivariant deformation retract]
\label{biequivariant}
There exists a deformation retract $\Phi:[0,1]\times\dif^{\infty}_+(C)\to\Rot$ which is bi-equivariant for the action of $\Rot$, i.e. for every $(R_1,\gamma,R_2)\in\Rot\times\dif^{\infty}_+(C)\times\Rot$, and $s\in[0,1]$,
$$\Phi_s(R_1\circ\gamma\circ R_2)=R_1\circ\Phi_s(\gamma)\circ R_2.$$
\end{theorem}

\begin{proof}
In \cite{Schwartz1992}, Schwartz constructs a deformation retract $\Phi^0$ from $\dif^{\infty}_+(C)$ to $\Mob(C)$ which is left-equivariant for the action of $\Mob(C)$: for every $\gamma\in\dif^{\infty}_+(C)$, $T\in\Mob(C)$, and $s\in[0,1]$, $\Phi_s^0(T\circ\gamma)=T\circ\Phi_s^0(\gamma)$. It occurs that this deformation retract is also right-equivariant by the action of $\Rot$: for every $\gamma\in\dif^{\infty}_+(C)$, $Q\in\Rot$, and every $s\in[0,1]$, we have $\Phi_s^0(\gamma\circ Q)=\Phi_s^0(\gamma)\circ Q$.

Indeed, in order to construct the deformation retract, Schwartz considers smooth maps $\Gamma:[0,\infty)\times\R/2\pi\Z\to C$, with $\Gamma(t,.)\in\dif^{\infty}_+(\R/2\pi\Z;C)$, evolving according to the following PDE:
\begin{equation}
\label{schwarzianevol}
\partial_t\Gamma(t,x)=-\partial_x (\cS(\Gamma)(t,x))\partial_x\Gamma(t,x),
\end{equation}
where $\cS$ denotes the Schwarzian derivative. We recall that $\cS(T)=0$ for every $T\in\Mob(C)$, and that $\cS$ satisfies the following cocycle relation:
$$\cS(\gamma_1\circ\gamma_2)=(\gamma_2')^2 \cS(\gamma_1)\circ\gamma_2+\cS(\gamma_2),$$
in such a way that for every $(T,\gamma,\tau)\in \Mob(C)\times\dif_+^{\infty}(\R/2\pi\Z;C)\times\Trans$, we have:
\begin{equation}
\label{invarianceschwarzian}
\cS(T\circ\gamma)=\cS(\gamma)\,\,\,\,\,\,\,\,\,\,\,\,\,\,\textrm{and}\,\,\,\,\,\,\,\,\,\,\,\,\,\,\cS(\gamma\circ \tau)=\cS(\gamma)\circ \tau.
\end{equation}

Schwartz showed that the problem of finding $\Gamma$ satisfying the PDE \eqref{schwarzianevol} with a prescribed initial condition $\Gamma(0,.)=\gamma$ has a unique solution, and that moreover this solution approaches in the $C^{\infty}$-topology as $t\to\infty$ a unique function of the form  $\rho_{\gamma}\circ E$, where $\rho_\gamma$ is Möbius transform and we recall that $E(t)=e^{it}$ denotes the exponential map. By the Invariance Relations \eqref{invarianceschwarzian}, and the uniqueness of the solution of the Cauchy problem, it comes that for every $(T,\gamma,\tau)\in \Mob(C)\times\dif_+^{\infty}(C)\times\Rot$ and every $t\in[0,\infty)$, $\Gamma(t,T\circ\gamma\circ \tau)=T\circ\Gamma(t,\gamma)\circ Q$. In particular $\rho_{T\circ\gamma\circ \tau}\circ E=T\circ\rho_{\gamma}\circ E\circ\tau$.

By conjugating this flow by the exponential map, Schwartz gets a deformation retract from $\dif^{\infty}_+(C)$ to $\Mob(C)$ which is shown to be bi-equivariant for the action of $\Rot$. Recall that there is an identification $C\simeq\R\PP^1$ which identifies $\Mob(C)$ with $PSL_2(\R)$ and $\Rot$ with $PSO(2)$.  Now 
the \emph{polar decomposition} provides a deformation retract from $\Mob(C)$ to $\Rot$ which is also bi-equivariant.

Recall that for a matrix $A\in SL_2(\R)$, there is a unique pair of matrices $(Q,S)$ such as $A=QS$, where $Q\in SO(2)$ and $S$ is a positive definite symmetric matrix with determinant $1$. By the \emph{Spectral Theorem}, such a matrix writes as $\Omega^T\diag(\lambda,\lambda^{-1})\Omega$, for a matrix $\Omega\in SO(2)$, and a positive number $\lambda\geq 1$ (we use the notation $\diag(a_1,a_2)$ for the $2\times 2$ diagonal matrix with entries $a_1$, $a_2$). We have naturally a retraction sending $A$ to $Q$, defined by $\widetilde{\Phi}^1_s(A)=Q\Omega^T\diag(\lambda^{1-s},\lambda^{s-1})\Omega$.

This retraction is clearly bi-equivariant by the action of $SO(2)$, and thus passes to the quotient by $A\mapsto -A$: it gives a retraction $\Phi^1:[0,1]\times PSL_2(\R)\to PSO(2)$ which is bi-equivariant for the action of $PSO(2)$. From this, we deduce the desired bi-equivariant retraction from $\Mob(C)$ to $\Rot$.

Concatenating these two retractions, we get a bi-equivariant deformation retract from $\dif^{\infty}_+(C)$ to $\Rot$.
\end{proof}

\paragraph{Deformation of the tangent bundle.} We can now use the deformation retract in order to get our continuous symmetric path of metrics.

Consider $\gamma=\gamma_g\in\dif^{\infty}_+(C)$ the arc length parametrization defined with $\gamma(1)=1$.  Denote by $R_{\gamma}=\Phi_1(\gamma)$. Define the identity isotopy on $\A$ in polar coordinates by
$$\fhi_s(r e^{i t}) = r\Phi_{(1-s)\rho(r)}(\gamma)\circ R_{\gamma}^{-1}(e^{i t}).$$
where $s \in [0,1]$ and $\rho$ is the bump function defined in Paragraph \ref{auxiliarystep}.

Note that this identity isotopy is independent of the choice of a given arc length parametrization, and depends only on the metric $g$. Indeed, another arc length parametrization writes as $\gamma_1=\gamma\circ Q$ for some rotation $Q$. It is enough to note that by the equivariance property of $\Phi$, we have that $R_{\gamma_1}=R_{\gamma}\circ Q$ and  $\Phi_{(1-s)\rho(r)}(\gamma_1)=\Phi_{(1-s)\rho(r)}(\gamma)\circ Q$.

Note moreover that each $\fhi_s$ preserves the circles $\lbrace r = \text{constant}\rbrace$.  

\begin{lemma}
\label{tangent}
The parametrization of $C$ by Euclidean arc length is geodesic for the pullback metric $g_1 = \varphi_1^{\ast}g$. Moreover, if we set $A_1(s,g)=\fhi_s^*g$ for $s\in[0,1]$, we have $R_* A_1(s,g)=A_1(s,R_*g)$ for every $s\in[0,1]$.
\end{lemma}

\begin{proof}
The first part is a direct consequence of the definition of the function $\fhi_s$ above. It remains to prove the symmetry property.

Consider the metric $R_* g$: any length of arc parametrization of $R_* g$ reads as $R\circ\gamma\circ Q$, where $Q$ is some rotation. Now notice that:
\begin{eqnarray*}
R\circ\fhi_s(re^{i\theta})&=&r R\circ\Phi_{\rho(r)(1-s)}(\gamma)\circ R_{\gamma}^{-1}(e^{i\theta})\\
                            &=&r\Phi_{\rho(r)(1-s)}(R\circ\gamma\circ Q)\circ(R R_{\gamma} Q)^{-1} (Re^{i\theta})\\
														&=&\fhi_{s}'\circ R(re^{i\theta}),
\end{eqnarray*}
where $\fhi_s'$ is the homotopy corresponding to $R_*g$. We conclude that the symmetry property holds.
\end{proof}

\begin{remark}
The isotopy $\varphi_s$ described above depends continuously on $\gamma_g$. In particular it varies continuously with respect to $g \in H(\A)$ in the smooth topology.
\end{remark}

\begin{remark}
If the metric $g$ is standard around $C$ then $\gamma_g$ is the identity, and the isotopy is trivial (i.e. $\varphi_s$ is the identity for all $s$).  In particular $g_1 = g$ in this case.
\end{remark}

\subsubsection{Deformation of the normal bundle}

The next step is to perform a twist isotopy on the hyperbolic metric $g_1$ in order to render the Euclidean normal vector field to $C$ perpendicular to $C$ with respect to our deformed hyperbolic metric.

\paragraph{The isotopy.} Consider $N:C \to \R^2$ the unit inward normal vector field for the metric $g_1$ and let $a,b: C \to \R$ be such that
\[N(z) = -a(z)z + b(z)i z\]
for all $z \in C$.  Notice that $a$ is never equal to zero since $N$ is everywhere orthogonal to the vector field $z\mapsto i z$.

Let $\beta: \R \times C \to C$ be the flow of the vector field $X(z)=\frac{b(z)}{a(z)}i z$.  We define an identity isotopy on $\A$ in polar coordinates by
$$\chi_s(re^{i t}) = r\beta_{s\rho(r)(1-r)}(e^{i t})$$
where $\rho$ is as in the previous subsection, and $s \in [0,1]$.

\begin{lemma}
\label{orthogonalize}
Let $g_2 =\chi_1^{\ast}g_1$ be the pullback metric. Then the Euclidean normal vector field to $C$ is orthogonal to $C$ with respect to $g_2$ and the parametrization of $C$ by Euclidean arc length is a geodesic for $g_2$.
\end{lemma}
\begin{proof}
The second claim follows because $\chi_s$ is the identity on $C$ for all $s$.

Notice that $\chi_1(re^{i t}) = r\beta_{1-r}(e^{i t})$ on some neighborhood of $C$.  Hence the radial derivative of $\chi_1$ at a point $z \in C$ is
$$\partial_r\chi_1(z) = z - X(z)=\frac{a(z)z - b(z)i z}{a(z)}=-\frac{1}{a(z)}N(z).$$

The first claim follows because the right hand side is normal to $C$ with respect to $g_1$.  
\end{proof}

\begin{remark}
The isotopy $\chi_s$ depends only on the normal vector of $g_1$ and hence depends continuously on $g_1$ in the smooth topology. 
\end{remark}

\begin{remark}
If the metric $g_1$ is standard in a neighborhood of $C$ then one has $b(z) = 0$ so that $\beta$ is the trivial flow.  Hence the isotopy just described leaves $g_1$ invariant in this case.
\end{remark}

\paragraph{Symmetry of the procedure.} It remains to prove that the isotopy we described is symmetric. This is the content of the following lemma:
\begin{lemma}
\label{symmetry1}
If we set for $s\in[0,1]$, $A_2(s,g_1)=(\chi_s)^* g_1$, then we have that $R_* A_2(s,g_1)=A_2(s,R_* g_1)$ for every $s\in[0,1]$.
\end{lemma}

\begin{proof}
Denote by $N',X',\beta'$ the objects associated to $R_{\ast}g_1$. First, note that $R$ preserves the circles centered at the origin. Hence, it is enough to prove that  one has: $R\circ\beta_s=\beta_s'\circ R$.

But we see quite easily that the relation $\beta_s'=R\circ\beta_s\circ R$ holds, first because $R^2=Id$, and then because by definition $N'=R_{\ast}N=-N\circ R$. Hence, it comes that $X'=-X\circ R=R_{\ast}X$: we deduce that $R$ conjugates the two flows.
\end{proof}

\subsubsection{Conformality}

The $g_2$-norm of any unit Euclidean tangent vector to $C$ is $\ell/2\pi$ where $\ell$ is the length of $C$ with respect to $g_2$.  However the $g_2$-norm of a Euclidean normal vector to $C$ varies from point to point.  The purpose of this subsection is to deform $g_2$ to a metric $g_3$ such that both Euclidean normal and tangent vectors have the same $g_3$-norm.

\paragraph{The isotopy.} For this purpose let $N:C \to \R^2$ be the inward pointing vector field which has norm $\ell/2\pi$ and is orthogonal to $C$ with respect to $g_2$.  Fix $\lambda:C \to (0,\infty)$ so that $N(z) = -\lambda(z)z$ and define a smooth identity isotopy on $\A$ in polar coordinates by
$$\psi_s(re^{i t})= f_{s\lambda(e^{i t})}(r)e^{i t},$$
where $f_{\lambda}$ is the push diffeomorphism defined in Paragraph \ref{auxiliarystep}.

\begin{lemma}
The pullback metric $g_3 =\psi_1^{\ast}g_2$ is conformal with respect to the Euclidean metric on $C$.
\end{lemma}
\begin{proof}
Since $\psi_1$ is the identity on $C$ one obtains that the Euclidean unit tangent vector to $C$ has $g_3$-norm $\ell/2\pi$.

On the other hand the image of the Euclidean unit normal vector to $C$ at any point $z \in C$ is exactly $N(z)$ which has $g_2$-norm equal to $\ell/2\pi$ and is perpendicular to $C$ with respect to $g_2$. This establishes the claim.
\end{proof}

\begin{remark}
The isotopy $\psi_s$ depends smoothly on the metric $g_2$.
\end{remark}

\begin{remark}
If $g_2$ is standard on a neighborhood of $C$ then $\lambda$ is constant and equal to $1$, hence the isotopy is the identity in this case.
\end{remark}

\paragraph{Symmetry of the procedure.} It remains to prove that the procedure we described is symmetric.

\begin{lemma}
\label{symmetry3}
If we set for $s\in[0,1]$, $A_3(s,g_2)=(\psi_s)^* g_2$, then we have that $R_* A_3(s,g_2)=A_3(s,R_* g_2)$ for every $s\in[0,1]$.
\end{lemma}

\begin{proof}
Denote by $N',\lambda',\psi'_s$ the objects associated to $R_{\ast} g_2$. Here again, it is enough to prove that $R\circ\psi_s=\psi'_s\circ R$. By definition, we have $N'=\R_{\ast}N=-N\circ R$, in such a way that $\lambda'=\lambda\circ R$: this enough to ensure the conjugacy formula.
\end{proof}

\subsubsection{Standardness\label{sectionstandardness}}

\paragraph{Collar Lemma and Fermi coordinates.}

Consider the standard metric $\sigma = \sigma_\ell$ around $C$ (where $\ell$ is the length of $C$ with respect to $g_3$) and notice that $g_3$ coincides with $\sigma$ on $C$. Define the \emph{collar function} as:

\begin{equation}
\label{collarfunction}
\omega(\ell)=\sinh^{-1}\left(\frac{1}{\sinh(\frac{\ell}{2})}\right).
\end{equation}

The \emph{Collar Lemma} (see \cite{Bu}) asserts that the function $f$ that we will describe below is well defined on the round annulus $\NN_{g_3}$ formed by the points $z$ of the standard hyperbolic annulus $\A_{\ell}$ with $\sigma$-distance to $C$ less than $d_{g_3}=\Min(\omega(\ell),\delta_{g_3})$, where $\delta_{g_3}$ is the distance between $\partial\A$ and the geodesic $C$. Note that this function $d_{g_3}$ varies continuously with the metric $g_3$.

The function $f:\NN_{g_3}\to \A$ is defined as follows. The circle $C$ is mapped onto itself by the identity (here we use that the arc length parametrization of $C$ by $g_3$ is Euclidean: in the general case, we would have to use a more general differomorphism). Given $z \in \NN_{g_3}\cap \CC^{\pm}$, where $\CC^{\pm}=\{z\in\C:|z|^{\pm 1}>1\}$, let $d(z)$ be the $\sigma$-distance between $z$ and $z/|z|$. Then $f(z)$ is the point of $\CC^{\pm}$ which is at $g_3$-distance $d(z)$ along the $g_3$-geodesic passing through $z/|z|$ perpendicular to $C$. We just described the so-called \emph{Fermi coordinates} \cite{Bu}.

Note that a priori the Fermi coordinates depend on the choice of an arc length parametrization of $C$ by a rotation. Since $\sigma$ is rotationally invariant, the definition of $g$ as $F_*\sigma$ is coherent.

\paragraph{Interpolation between the identity and the Fermi coordinates.} Let $\kappa$ be the maximum of the first order derivatives of the coefficients of the metric $g_3$ on $\A$.  The constants $\epsilon_i$ below can be chosen to be continuous functions of $\ell$ and $\kappa$.

Let $\epsilon_2 \l \epsilon_1 \l 1/10$ be positive and such that the following properties hold.
\begin{enumerate}
 \item the metrics $\sigma$ and $g_3$ are bi-Lipschitz with Lipschitz constant less than or equal to $2$ on the annulus $\A_1 = \lbrace 1-\epsilon_1 \le |z| \le 1+\epsilon_1\rbrace$ (use here that the two metrics coincide in $C$).  In particular $\sigma$ is defined on this annulus.
 \item The $\sigma$ distance between the two boundaries of $\A_2 = \lbrace 1-\epsilon_2 \le |z| \le 1+\epsilon_2\rbrace$ is less than or equal to half of the corresponding distance between the boundaries of $\A_1$.
 \item The $\sigma$ distance between the two boundaries of $\A_2$ is less than or equal to $\omega(\ell)$ in such a way that $f$ is well defined on $\A_2$.
\end{enumerate}

Notice that $f$ is the identity on $C$ and furthermore, because $\sigma$ and $g_3$ coincide on $C$, one has that the differential $Df$ is the identity map at all points of $C$.  Furthermore $f$ is an isometry between $\sigma$ restricted to $\A_2$ and $g_3$ restricted to $f(\A_2) \subset \A_1$ (see the first two items).

For some $\epsilon \l \epsilon_2$ to be chosen later we define:
$$\rho_{\epsilon}(r)=\rho\left(1+\frac{r-1}{\epsilon}\right),$$
where $\rho$ is the bump function of Paragraph \ref{auxiliarystep}. Note that $z \mapsto\rho_{\epsilon}(|z|)$ is $1$ in a neighborhood of the unit circle and is $0$ outside $\A^{\epsilon}=\{z\in\C:\,1-0.09\,\epsilon<|z|<1+0.09\,\epsilon\}$. We will assume that $\epsilon$ is small enough in such a way that the $\epsilon$-neighborhood of $\A^{\epsilon}$ is included in $\A$.

Define for $s\in [0,1]$ the following map $\A\to\C$.
\[\Omega_s(z) = z + s\rho_{\epsilon}(|z|)(f(z) - z)\]
Outside of $\A_2$ the function $f$ might be undefined, but $\rho_{\epsilon}(|z|) = 0$, so it is understood that $\Omega_s(z) = z$ at these points for all $s \in [0,1]$.

The $1$-jets of $f$ and of the identity coincide on $C$, thus $|f(z)-z|$ is controlled by some quantity of order $\epsilon^2$ on $\A^{\epsilon}$ (the precise constants depending on $\kappa$). In particular, when $\epsilon$ is small enough, the segment $[z,f(z)]$ is included in the $\epsilon$-neighborhood of $A^{\epsilon}$ and $\Omega_s$ defines a map $\A\to \A$.

\paragraph{Properties of the interpolation.} Using again that $g_3$ is confomal at $C$, we can show that each step of the interpolation yields a diffeomorphism between $\A$ and its image.

\begin{lemma}
For every $s\in[0,1]$, $\Omega_s:\A\to\C$ is a diffeomorphism on its image.
\end{lemma}

\begin{proof}
Let $\kappa_2$ be an upper bound for the derivative of $\rho$ and notice that the differential of $z \mapsto \rho_{\epsilon}(|z|)$ is bounded by $\kappa_2/\epsilon$. Using again that $|f(z)-z|$ is of the order of $\epsilon^2$ on $\A^{\epsilon}$, we see that we may choose $\epsilon$ (continuously depending on $\ell$,$\kappa$ and $\kappa_2$) so that for all $z \in \A$:
\[|D\Omega_s(z) - \text{Id}| \le \kappa_2|f(z)-z| + \rho_{\epsilon}(z)|Df(z) - \text{Id}| \le 1/2.\]

For this value of $\epsilon$ one obtains that $D\Omega_s$ is everywhere invertible for all $s$.  Since $\Omega_s$ is the identity in $C$, it must have degree $1$ and therefore is a diffeomorphism for all $s$.
\end{proof}

\begin{lemma}
The pullback metric $g_4 =\Omega_1^{\ast}g_3$ is standard around $C$.
\end{lemma}
\begin{proof}
The claim amounts to the fact that $\Omega_1$ coincides with $f$ on a neighborhood of $C$.
\end{proof}

\begin{remark}
The isotopy $\Omega_s$ depends smoothly on the metric $g_3$ through the constants $\ell$ and $\kappa$.
\end{remark}

\begin{remark}
Even if $g_3$ is standard on a neighborhood of $C$ it may be the case that $\Omega_s$ is not the identity map for all $s$ (and even $\Omega_1$ may not be the identity).  However, in this case, $\Omega_s$ will be the identity on a neighborhood of $C$ for all $s \in [0,1]$.

%
\end{remark}

\begin{lemma}
\label{preservesym4}
If we set for $s\in[0,1]$, $A_4(s,g_3)=(\Omega_s)^* g_3$, then we have that $R_* A_4(s,g_3)=A_4(s,R_* g_3)$ for every $s\in[0,1]$.
\end{lemma}

\begin{proof}
Let $f'$, and $\Omega_s'$ be the objects associated to $R_{\ast} g_3$. The definition of Fermi coordinates implies that $R$ conjugates $f$ and $f'$, hence it conjugates the homotopies $\Omega_s$ and $\Omega_s'$. Finally, the equality  $R_* A_4(s,g_3)=A_4(s,R_* g_3)$ has to hold.
\end{proof}

%% file: pants.tex
\subsection{Hyperbolic metrics and conformal structures}
\input{hyperbolicmetrics.tex}
\subsection{Diffeomorphisms between hyperbolic metrics}
\input{canonicaldiffeomorphisms.tex}
\subsection{Boundary admissible hyperbolic metrics}
\input{admissible.tex}
\subsection{Admissible and symmetric sections of Teichmüller space}
\input{globalsections.tex}

\subsection{Homotopy to an admissible and symmetric section}
\input{homotopytosection.tex}

%% file: hyperbolicmetrics.tex
\paragraph{Planar pairs of pants.} A \emph{planar pair of pants} is an ordered triple of open Euclidean disks $(D_{out},D_{left},D_{right})$ such that the closures of $D_{left}$ and $D_{right}$ are disjoint and are both contained in the interior of $D_{out}$.   We sometimes identify a planar pair of pants with the subset of the plane defined by $\overline{D_{out}} \setminus (D_{left} \cup D_{right})$ where the overline denotes closure (however it is important for us to keep track of which of the interior disks is the ``left'' one and which is the ``right'').

\begin{lemma}\label{conformalmapsofpants}
Let $P$ and $Q$ be planar pairs of pants and $f:P \to Q$ be a conformal mapping between them.  Then $f$ is the restriction of a Möbius transformation. 
\end{lemma}
\begin{proof}
Let $G_P$ be the group generated by inversions with respect to the boundary components of $P$ and $G_Q$ the corresponding group for the pair of pants $Q$.   If we let $K_P$ and $K_Q$ be the limit sets of $G_P$ and $G_Q$ respectively there is a unique extension $f$ to a map $F:\widehat{\C} \setminus K_P \to \widehat{\C}\setminus K_Q$ satisfying $F\circ I = I'\circ F$ for each inversion $I$ with respect to a boundary component of $P$, where $I'$ is the inversion with respect to the image under $f$ of this component.

Because $F$ conjugates the actions of $G_P$ and $G_Q$ it extends continuously in a unique way to $K_P$ yielding a homeomorphism of the Riemann sphere which is conformal outside of $K_P$.

However, since $K_P$ has zero one dimensional Hausdorff measure one obtains that $F$ extends conformally to the entire Riemann sphere by Painlevé's theorem (e.g. see \cite[Theorem 2.7]{dudziak2010}).  Therefore $f:P \to Q$ is the restriction of a Möbius transformation.
\end{proof}

\paragraph{Beltrami coefficients.} Let $H(P)$ be the space of smooth hyperbolic Riemannian metrics on $P$ with the property that all three boundary circles are geodesics.  The space is endowed with the topology of smooth convergence.  Using complex notations (i.e. writing $\d z=1/2(\d x+i\d y)$ and $\d\bar{z}=1/2(\d x-i\d y)$: see \cite{Ahlfors2006}), each such Riemannian metric $g$ can be written in the form:
\[\d s^2 = \lambda |\d z + \mu \d \overline{z}|^2\]
for a unique pair $\lambda$ and $\mu$ of smooth functions from $P$ to $\C$ satisfying $\lambda \in \R_+^{\ast}$ and $|\mu| \l 1$ at all points. The functions $\lambda$ and $\mu$ are smooth functions on the coefficients of $g$.

The smooth function $\mu$ is called the \emph{Beltrami coefficient} of the metric $g$.  We denote the space of smooth Beltrami coefficients (i.e. smooth complex valued functions with modulus strictly less than $1$) endowed with the topology of smooth convergence by $B(P)$.  Each Beltrami coefficient can be interpreted as a conformal structure on $P$ via the formula $\d s^2 = |\d z + \mu \d \overline{z}|^2$.

Notice that while $B(P)$ is convex it is not self-evident that $H(P)$ is even arc-connected (how does one interpolate between two hyperbolic metrics with geodesic boundary while preserving these properties?).  The following theorem allows one to construct continuous paths of hyperbolic metrics with geodesic boundary components on $P$ by going through $B(P)$.

\begin{theorem}[Equivalence of conformal structures and hyperbolic metrics]\label{hyperbolicvsbeltrami}
The map $g \mapsto \mu$ associating to each hyperbolic metric in $H(P)$ its Beltrami coefficient in $B(P)$ is a homeomorphism.
\end{theorem}
\begin{proof}
The result amounts to establishing that there is a unique hyperbolic metric with geodesic boundary in the conformal equivalence class determined by each Beltrami coefficient.

For compact surfaces without boundary the equivalence between Beltrami coefficients and conformal equivalence classes of Riemannian metrics is established for example in \cite[Theorem 1.8]{ImayoshiTamagoshi}.  The uniformization theorem implies that each conformal class of Riemannian metrics contains exactly one hyperbolic metric (see for example \cite{Berger1971}).

For the pair of pants $P$ one way to proceed is to adapt Berger's PDE argument (adding a Neumann boundary condition) to establish that there is exactly one hyperbolic metric with geodesic boundary in each conformal equivalence class.

Another approach is to extend the Beltrami coefficient smoothly to the plane and solve the \emph{Beltrami equation} (see \cite{Ahlfors2006}).  This maps the given conformal structure on $P$ to the Euclidean structure on some domain in $\C$ with smooth boundary.  Applying Koebe's theorem this domain can be mapped conformally to a domain whose boundary curves are circles.  At this point one can ``double'' the surface by using inversions with respect to the three boundary circles.  After this is done the result is reduced to the case of a compact surface without boundary (the only caveat being that if the given conformal structure possesses a conformal symmetry then that symmetry is an isometry for the hyperbolic metric given by the uniformization theorem).
\end{proof}

A conformal metric on a planar pair of pants $P$ is a smooth Riemannian metric with Beltrami coefficient equal to $0$ (i.e. the metric is conformal to the Euclidean metric on $\C$).  We will need the following corollary of the theorem above:
\begin{coro}\label{planarpantsconformalmetric}
Every planar pair of pants $P$ admits a unique conformal hyperbolic metric $g_P$ with geodesic boundary.
\end{coro}

%% file: canonicaldiffeomorphisms.tex
\paragraph{The Teichmüller space.} From now on we let $P$ be the particular planar pair of pants with $D_{out} = \lbrace |z| \l 1\rbrace, D_{left} = \lbrace |z+1/2| \l 1/4\rbrace$ and $D_{right} = \lbrace |z-1/2| \l 1/4\rbrace$.  We call the boundaries of the three disks, the outer, left and right boundaries respectively.

The space $\diffid(P)$ of smooth self-diffeomorphisms of $P$ which are smoothly isotopic to the identity (endowed with the topology of smooth convergence) acts continuously on both $B(P)$ and $H(P)$ by pushforward and, in fact, coincides with the space of diffeomorphisms which leave each boundary component invariant. See \cite{Ahlfors2006} for the definition of the action of diffeomorphisms on Beltrami coefficients.  The actions on $B(P)$ and $H(P)$ are conjugated by the homeomorphism of Theorem \ref{hyperbolicvsbeltrami}.

We say that two metrics in $H(P)$ are equivalent if one is a pushforward of the other via a diffeomorphism in $\diffid(P)$.  The space of equivalence classes $H(P)/\diffid(P)$ is by definition the \emph{Teichmüller space} $T(P)$ of the pair of pants $P$.

\paragraph{Product structure.} Given a metric $g \in H(P)$ let $L(g) = (\ell_{out}(g),\ell_{left}(g),\ell_{right}(g))$ be the triple of lengths (with respect to $g$) of the outer, left, and right boundaries respectively.   Notice that pushing forward $g$ via a boundary preserving diffeomorphism $f$ does not change the value of $L$ (i.e. $L(f_*g) = L(g)$) and hence $L$ is a well defined function on $T(P)$.  In fact $L$ is a homeomorphism between $T(P)$ and $(\R_+^{\ast})^3$ and furthermore we have the following theorem (see \cite[Section 4, Corollaries 1 and 2]{ES}):

\begin{theorem}[Product structure of the space of hyperbolic metrics]\label{productstructure}
There exists a homeomorphism $\varphi:H(P) \to (\R_+^{\ast})^3 \times \diffid(P)$ satisfying the following two properties:
\begin{enumerate}
 \item For all $g \in H(P)$ and $f \in \diffid(P)$ one has $\varphi(f_{*}g) = (L(g),f\circ h)$ where $\varphi(g) = (L(g),h)$. 
 \item The following diagram commutes:
\begin{tikzcd}
H(P) \arrow{r}{\varphi} \arrow{d}{\pi}
& (\R_+^{\ast})^3 \times \diffid(P) \arrow{d}{\pi_1}\\
T(P) \arrow{r}{L} &(\R_+^{\ast})^3
\end{tikzcd}
\end{enumerate}
\end{theorem}

The proof of the theorem relies on a continuity theorem for the solution of the Beltrami equation with respect to the coefficient, which itself is based on the Ahlfors-Bers theorem \cite{ahlfors-bers1960}. We will need the following consequence of the above theorem:
\begin{coro}[Canonical diffeomorphisms between equivalent metrics]
If $g,g' \in H(P)$ are equivalent then there exists a unique diffeomorphism $f_{(g,g')} \in \diffid(P)$ such that $f_{(g,g')*}g = g'$.  Furthermore the function $(g,g') \mapsto f_{(g,g')}$ is continuous.
\end{coro}

We will use the diffeomorphisms above to construct identity isotopies from paths between equivalent metrics in the following way.
\begin{coro}[Paths of equivalent metrics yield isotopies]\label{isotopiesvscurvesofmetrics}
If $t \mapsto g_t$ is a continuous path of equivalent metrics then $t \mapsto f_{(g_0,g_t)} = f_t$ is the unique identity isotopy with $f_{t*}g_0 = g_t$ for all $t$.
\end{coro}

%% file: admissible.tex
We say that a metric $g \in H(P)$ on the pair of pants $P$ is \emph{boundary admissible} (or just \emph{admissible}) if it is standard around each boundary of $P$.  

Using the Massage procedure described in Section \ref{massageannuli} near the boundary of $P$, it is possible to deform any hyperbolic metric in $P$ with geodesic boundary to an admissible one.

\begin{theorem}\label{admissiblehomotopy}
There exists a continuous function $F:[0,1]\times H(P) \to H(P)$ such that
\begin{enumerate}
 \item For all $g \in H(P)$ one has that $F(0,g) = g$, $F(1,g) \in A(P)$, and $F(t,g)$ is equivalent to $g$ for all $t \in [0,1]$.
 \item If $g \in A(P)$ then $F(t,g) \in A(P)$ for all $t \in [0,1]$.
 \item It is symmetric: for all $g \in H(P)$, we have $R_* F(t,g) = F(t,R_*g)$.
\end{enumerate}
\end{theorem}

\begin{proof}
This theorem follows directly from the Massage Lemma (see Theorem \ref{massagelemma}). Indeed, recall that we stated $\A=\{z\in\C:\,0.9<|z|<1.1\}$. The constants have been chosen in such a way that $\A$ and its images by the affine maps $h_{\pm}:z\mapsto z/4\pm 1/2$ (which transform the outer boundary of $P$ in the other boundary components) of $\A$ form three disjoint neighborhoods of the boundary components of $P$.

Now since the isotopy constructed in Theorem \ref{massagelemma} preserves the unit circle, it also preserves the two semi-annuli $\A^+=\{z\in \A:\,|z|\geq 1\}$ and $\A^-=\{z\in \A:\,|z|\leq 1\}$. Hence, we can perform the isotopy in the semi-annuli around the boundary components, in order to get the desired function $F$.

It remains to check that the procedure we described is symmetric. The metric $g$ yields one hyperbolic metric in $\A^-$, and two on $\A^+$  that we denote by $g_{out},g_{left}$ and $g_{right}$ (the last two correspond respectively to the pullback by $h_-$ and $h_+$ of $g$ near the left, right leg).

Note that $R\circ h_-=h_+\circ R$: it comes that the corresponding objects for $R_*g$ are $R_* g_{out}$, $R_*g_{right}$ and $R_* g_{left}$.

Using the symmetry of the maps $A(s,.)$ (see Theorem \ref{massagelemma}), we get the symmetry of the procedure.

\end{proof}

%

%% file: globalsections.tex
\paragraph{Global sections of Teichmüller space.} By a \emph{continuous global section} of $T(P)$ we mean a continuous mapping from $(\R_+^{\ast})^3$ to $H(P)$ associating to each triple one of its preimages with respect to $L\circ \pi$ (where $\pi$ is the quotient projection from $H(P)$ to $T(P)$).  The existence of continuous global sections follows immediately from Theorem \ref{productstructure}.

\begin{coro}[Continuous global sections]\label{globalsection}
There exists a continuous global section of $T(P)$.
\end{coro}

\paragraph{Symmetric sections.} We will need to improve the above result.   Let $R:P \to P$ be the rotation of angle $\pi$ (i.e. $R(z) = -z$), we say that a section $\ell = (\ell_{out},\ell_{left},\ell_{right}) \mapsto g_{\ell}$ is \emph{symmetric} if $g_{\ell} = R_* g_{\sigma(\ell)}$ for all $\ell$, where we define $\sigma(\ell_{out},\ell_{left},\ell_{right}) = (\ell_{out},\ell_{right},\ell_{left})$ for all $\ell \in (\R_+^{\ast})^3$.

In particular for a symmetric section $\ell \mapsto g_{\ell}$ each metric of the form $g_{(\ell_{out},t,t)}$ must have $R$ as a self-isometry.

\begin{theorem}[Continuous symmetric global section]\label{symmetricglobalsection}
There exists a continuous symmetric global section of $T(P)$.
\end{theorem}
\begin{proof}
Let $H(P)^R$ be the subset of $H(P)$ consisting of metrics for which $R$ is an isometry, $T(P)^R$ be the subset of $T(P)$ consisting of metrics which have an isometry isotopic to $R$, and $\diffid(P)^R$ be the subset of $\diffid(P)$ consisting of diffeomorphisms which commute with $R$.  By \cite[Theorems 5B,5C, and 5D]{ES} when endowed with the projection $\pi:H(P)^R \to T(P)^R$ the space $H(P)^R$ is a principal $\diffid(P)^R$ fiber bundle.

Any principal bundle over a contractible base is trivial.  Hence, it suffices to show that $T(P)^R$ is homeomorphic via $L$ to $S = \lbrace (\ell_{out},\ell_{left},\ell_{right}) \subset (\R_+^{\ast})^3: \ell_{left} = \ell_{right}\rbrace$ to conclude that $L\circ \pi:H(P)^R \mapsto S$ admits admits a global continuous section.

Clearly the image of $T(P)^R$ under $L$ is included in the subspace $S$.  To see that in fact $L(T(P)^R) = S$ one must construct for each triple $\ell = (\ell_{out},t,t)$ a hyperbolic metric $g$ on $P$ with $L(\pi(g)) = \ell$ and such that $R$ is an isometry for $g$.  We postpone this until Lemma \ref{symmetricmetrics}.

Given a global section $(\ell_{out},t,t) \mapsto g_{(\ell_{out},t,t)}$ of $T(P)^R$ one can use the homeomorphism of Theorem \ref{productstructure} to identify it with a continuous function from $S$ to $\diffid(P)$.   Any such function can be extended continuously to the half space $\lbrace \ell_{left} \ge \ell_{right}\rbrace \subset (\R_+^{\ast})^3$  by composing it with a retraction (e.g. the orthogonal projection) from the half space to $S$.  After this, one may extend it continuously to all of $(\R_+^{\ast})^3$ by symmetry (i.e. defining $g_{(\ell_{out},\ell_{left},\ell_{right})} = R^*g_{(\ell_{out},\ell_{right},\ell_{left})}$) thus obtaining a continuous symmetric global section.
\end{proof}

\paragraph{Construction of symmetric metrics.} We now show how to construct symmetric metrics on $P$ for each Teichmüller class which might admit one.
\begin{lemma}\label{symmetricmetrics}
For each $\ell = (\ell_{out},t,t) \in (\R_+^{\ast})^3$ there exist a metric $g \in H(P)$ with $L(\pi(g)) = \ell$ and such that $R$ is an isometry for $g$.
\end{lemma}
\begin{proof}
Let $g \in H(P)$ be any preimage of $\ell$ and let $\mu$ be its Beltrami coefficient.  By solving the Beltrami equation $\partial_{\bar{z}} f = \mu \partial_z f$ for any smooth extension of $\mu$ to the entire plane one can map $P$ endowed with the conformal structure given by $\mu$ conformally onto a planar domain $\Omega$ bounded by smooth curves with the conformal structure coming from $\C$.  Koebe's theorem implies that $\Omega$ can be mapped conformally onto a planar pair of pants $Q$.

Hence we have obtained a diffeomorphism $f:P \to Q$ which is an isometry between $(P,g)$ and a planar pair of pants $Q$.  Since, by Corollary \ref{planarpantsconformalmetric}, $Q$ admits a unique conformal hyperbolic metric with geodesic boundary $g_Q$ it must be the case that $f_*g = g_Q$.   Composing with a Möbius transformation we can assume that $Q$ is defined by $D_{out} = \lbrace |z| \l 1\rbrace, D_{left} = \lbrace |z+x| \l r_1\rbrace$ and $D_{right} = \lbrace |z-x| \l r_2\rbrace$ for some $x \in (0,1)$ and $r_1,r_2 \l 1$.

We claim that $r_1 = r_2$.  Notice that the metric $g$ on $P$ is equivalent to its pushforward $R_*g$ under the rotation of angle $\pi$.  Hence $(P,g)$ admits a self isometry which exchanges the left and right boundaries.   This in turn, implies that there is a conformal self-mapping of $Q$ which exchanges the left and right boundaries.  But by Lemma \ref{conformalmapsofpants} this self-mapping must be a Möbius transformation of the unit disk and the only way this is possible is if $r_1 = r_2$ and the conformal self mapping is $R(z) = -z$.

It is possible to construct a boundary preserving diffeomorphism $f:Q \to P$ which commutes with $R$.  To see this notice that it is clearly possible to define an isotopy on the boundary circle of $Q$ centered at $x$ with radius $r$ deforming it to the boundary circle of $P$ centered at $1/2$ with radius $1/4$.  We assume such an isotopy has been chosen so that the image of the initial circle never leaves the region $\lbrace (u,v) \in \R^2: u \ge \min(x-r,1/4)/2 , |u+iv| \le (\max(x+r,3/4)+1)/2\rbrace$.   Using the isotopy extension theorem one can extend this to an isotopy of the right half plane which is the identity outside of the previously defined region.  This istopy can then be extended to the left half plane in such a way that the resulting isometry commutes with $R$.  The endpoint of the the isotopy is the required diffeomorphism.

Given any such diffeomorphism the pushforward metric $f_{*}g_Q$ is isometric to $g$ via a boundary preserving (and hence isotopic to the identity) diffeomorphism and is $R$-invariant.
\end{proof}

\paragraph{Admissible and symmetric sections.} A section of $T(P)$ is said to be \emph{admissible} if it only takes values in the set $A(P)$ of admissible metrics.  Using the procedure described in Theorem \ref{admissiblehomotopy}, the following result follows:

\begin{theorem}\label{admissiblesymmetricglobalsection}
There exists a continuous admissible and symmetric global section of $T(P)$.
\end{theorem}

%% file: homotopytosection.tex
Our final task in this section is to establish the existence of a special section of Teichmüller space, which is not only admissible and symmetric, but comes with a procedure for deforming any given metric on $P$ to an equivalent metric in the section.   The deformation procedure respects symmetry under $R$ and is continuous with respect to the initial metric.

This special section of $T(P)$ will be useful both for the construction of a large family of `model hyperbolic metrics' on the Hirsch foliation (which we will carry out in the next section), and for deforming any given hyperbolic in the foliation to a model metric (at least in certain special cases).

\begin{theorem}\label{theperfectsection}
There exists a continuous homotopy $H:[0,1]\times H(P) \to H(P)$ starting at the identity and satisfying the following properties:
\begin{enumerate}
 \item For all $t \in [0,1]$ and $g \in H(P)$ the metric $H(t,g)$ is equivalent to $g$.
 \item For all $t \in [0,1]$ and $g \in H(P)$ one has $R_{*}H(t,g) = H(t,R_{*}g)$.  In particular if $g$ is invariant under $R$ then so is $H(t,g)$ for all $t$.
 \item If $g$ is boundary admissible then so is $H(t,g)$ for all $t \in [0,1]$.
 \item If $g,g' \in H(P)$ are equivalent then $H(1,g) = H(1,g')$.
 \item The section of Teichmüller space associating to each class its image under $H(1,\cdot)$ is continuous admissible and symmetric.
\end{enumerate}
\end{theorem}

The proof of this theorem consists in three steps. We first construct \emph{symmetric paths} of metrics to the \emph{unique} conformally Euclidean hyperbolic metric on $P$, $g_P$ (in particular the ending point of these paths is independent of the starting point). In particular, such a path does not preserve Teichmüller classes. But note that giving a continuous path in $H(P)$ is the same thing as giving a continuous path in $\diffid(P)\times T(P)$. Using the previous family of paths, as well as the symmetric section given by Theorem \ref {admissiblesymmetricglobalsection}, we find continuous paths of diffeomorphisms, which give paths of \emph{equivalent metrics}. We show that these path are again symmetric and that the time $1$ of the paths induces a symmetric section of the Teichmüller space that we characterize. Finally, projecting the paths via $F$ provides paths of metrics \emph{preserving admissible metrics} while keeping the other properties.

\subsubsection{Symmetry}

Given a metric $g\in H(P)$ let $\mu$ be its corresponding Beltrami coefficient and define the continuous path $h:[0,1] \times H(P) \to H(P)$ so that the Beltrami coefficient corresponding to $h(t,g)$ is exactly $(1-t)\mu$ (this is possible by Theorem \ref{hyperbolicvsbeltrami}).

Since $R$ is holomorphic with $\partial_zR=-1$, the Beltrami coefficient associated to $R_{*}g$ is $\mu\circ R$ for every $g\in H(P)$.  Hence both the Beltrami coefficients of $R_{*}h(t,g)$ and $h(t,R_*g)$ are $(1-t)\mu\circ R$. Using again Theorem \ref{hyperbolicvsbeltrami}, this implies that the two metrics $R_{*}h(t,g)$ and $h(t,R_*g)$ coincide: \emph{the path $h$ is symmetric}.

However boundary admissibility is not preserved by $h$ and neither are Teichmüller equivalence classes.   In fact one has $h(1,g) = g_P$ for all $g$, where $g_P$ is the unique hyperbolic conformal metric with geodesic boundary on $P$ given by Corollary \ref{planarpantsconformalmetric}.

\subsubsection{Preservation of Teichmüller classes}

We will first modify $h$ so that it preserves Teichmüller equivalence classes (while preserving its symmetry under $R$).

To do this let $\ell \mapsto g_{\ell}$ be a continuous symmetric global section of $T(P)$ (such an object exists by Theorem \ref{admissiblesymmetricglobalsection}).  Recall that the map $g \mapsto (L(g),f)$ is a homeomorphism between $H(P)$ and $(\R_+^{\ast})^3 \times \diffid(P)$, where $L(g) = (\ell_{out}(g),\ell_{left}(g),\ell_{right}(g))$ and $f$ is the unique element of $\diffid(P)$ such that $g = f_{*}g_{L(g)}$ (see Theorem \ref{productstructure}).

Given $g \in H(P)$ let $g_t = h(t,g)$, $\ell_t = L(g_t)$, and $f_t \in \diffid(P)$ be the unique diffeomorphism such that $g_t = f_{t*}g_{\ell_t}$.

We define $h'(t,g) = f_{t*}g_{\ell_0}$.   It is clear that $h'(t,g)$ is equivalent to $g$ for all $t \in [0,1]$.  We will show that $h'(t,R_*g) = R_*h'(t,g)$ for all $t \in [0,1]$ as well.

Notice that $L(R_*g) = \sigma(\ell_t)$ where $\sigma$ exchanges the `left' and `right' lengths of each triple.  By definition $h'(t,R_*g) = f'_{t*}g_{\sigma(\ell_0)}$ where $f'_t$ is the unique diffeomorphism in $\diffid(P)$ satisfying $f'_{t*}g_{\sigma(\ell_t)} = h(t,R_*g)$.

We claim that $f_t'=R\circ f_t\circ R$.

Indeed, using the symmetry of the section $\ell \mapsto g_\ell$ and of $h$ we obtain:
\[(R\circ f'_{t}\circ  R)_*g_{\ell_t} = R_* f'_{t*} g_{\sigma(\ell_t)}=R_* h(t,R_* g)=h(t,g).\]
By uniqueness of $f_t$, we must have $f_t=R\circ f_t'\circ R$. Putting all this together we obtain:
\[h'(t,R_*g) = (R\circ f_t\circ R)_{*}g_{\sigma(\ell_0)} = R_{*} f_{t*} g_{\ell_0} = R_{*}h'(t,g),\]
as desired.

Finally,  $h'(1,g)=f_{1*}g_{\ell(g)}$ where $f_1$ is characterized by $f_{1\ast}g_{\ell(g_P)}=g_P$. In particular, $h'(1,g)$ depends only on the Teichmüller class of $g$.

\subsubsection{Preservation of admissibility}

So far, we have defined a homotopy $h':[0,1]\times H(P) \to H(P)$ satisfying all the desired properties except that $h'(t,g)$ need not be admissible even if $g$ is, and in particular the section given by $h'(1,\cdot)$ is not admissible (though it is symmetric).

To fix this problem we use the homotopy $F:[0,1]\times H(P) \to H(P)$ given by Theorem \ref{admissiblehomotopy}.  Consider the continuous map $h'':[0,1]\times H(P) \to H(P)$ defined by

\[h''(t,g) = F(1,h'(t,g)).\]

This map provides a path between $F(1,g)$ and $F(1,h'(t,g))$. The concatenation of $F$ and $h''$ provides the desired homotopy $H:[0,1]\times H(P)\to H(P)$.

It follows from the properties of $h'$ and $F$ that $H(t,g)$ is equivalent to $g$ for all $t \in [0,1]$.  

Remember that $F$ is symmetric: for every $(t,g)\in [0,1]\times H(P)$, we have $R_{\ast}F(t,g)=F(t,R_{\ast}g)$. Hence since $h'$ is symmetric, we get that $h''$ is symmetric as well. Finally, we obtain $R_{\ast}H(t,g)=H(t,R_{\ast}g)$ for every $(t,g)\in [0,1]\times H(P)$

%

The fact that if $g$ is admissible then $H(t,g)$ is too for all $t \in [0,1]$ follows from Property 2 of $F$ in the statement of Theorem \ref{admissiblehomotopy}.

We had previously verified that $h'(1,g)$ depended only on the Teichmüller class of $g$ (equivalently only on $L(g)$): this property is again satisfied by $H$.

The fact that the section of Teichmüller space given by $H(1,\cdot)$ is admissible follows from the properties of $F$, and symmetry is because $R_*H(1,g) = H(1,R_*g)$ for all $g$.  This concludes the proof.

%% file: models.tex
\subsection{Construction of model metrics\label{modelmetrics}}

We will now use the notation given in Section \ref{hirschfoliation} which for the reader's convenience we recall briefly.  The Hirsch foliation was constructed by taking an explicit endomorphism $f:S^1 \times \C \to S^1 \times \C$ which had a solenoid attractor inside a solid torus $\fT = S^1 \times \D$ and considering the quotient manifold $M = (\overline{\fT} \setminus f(\fT))/f$.   The Hirsch foliation is the projection of the foliation on $S^1 \times \C$ by sets of the form $\lbrace e^{it}\rbrace \times \C$ and the leaves contains a family of pair of pants defined by $P_t = \overline{\fT} \setminus f(\fT) \bigcap \lbrace e^{it}\rbrace \times \C$.   One has $P_t = P_{t + 2\pi}$ and one can identify each $P_t$ with the standard pair of pants $P$ by projection onto $\C$ followed by the rotation $z \mapsto e^{-it/2}z$.   However, it is important to note that the identification one obtains for $P_t$ and $P_{t+2\pi}$ differ by a rotation $R$ of angle $\pi$.

\paragraph{Metrics in a fundamental domain.} Given a continuous \emph{length function} $\lambda:S^1 \to \R_+^{\ast}$ and a continuous \emph{twist function} $\tau:S^1 \to \R$ we will construct a hyperbolic metric on the Hirsch foliation.  For this purpose we fix for what remains of the article, a homotopy $H:[0,1]\times H(P) \to H(P)$ and a continous symmetric admissible section $\ell \mapsto g_{\ell}$ of $T(P)$ given by Theorem \ref{theperfectsection}.

Using the identification of each $P_t$ with the standard pair of pants $P$, we consider the on $P_t$ coming from the global section $\ell \mapsto g_{\ell}$ such that the outer boundary (the one corresponding to $S^1$ under the identification with $P$) has length $\lambda(e^{it})$ and the other two have lengths $\lambda(e^{it/2})$ and $\lambda(e^{i(t+2\pi)/2})$ according to whether their preimages under $f$ belong to $P_{t/2}$ or $P_{(t+2\pi)/2}$.  Let $g_t$ be the thus chosen metric on $P_t$.   Notice that because the section we have chosen is symmetric one has $g_t = g_{t+2\pi}$ so the metric is defined unambiguously.

\paragraph{Extension with a twist function.} We now use the continuous twist function $\tau:S^1\to\R$. Consider the \emph{twisted solenoid} $f_\tau:S^1 \times \C \to S^1 \times \C$ defined by:
\begin{equation}
\label{twistedsolenoid}
f_\tau(e^{it},z) = \left(e^{i2t},\frac{1}{2}e^{it} + \frac{e^{i\tau(e^{it})}}{4}z\right).
\end{equation}

Because the metric $g_t$ is admissible, its pushforward under $f_\tau$ (which is defined on a subset of $\lbrace e^{i2t}\rbrace \times \C$ contained in one of the `holes' of the pair of pants $P_{2t}$) coincides with $g_{2t}$ where both are defined.  Furthermore,  both metrics are standard around the circle where they intersect and therefore glue together smoothly.   Hence one may extend the family $g_t$ to a unique maximal $f_\tau$-invariant family of metrics on $(S^1 \times \C) \setminus K_\tau$ where $K_\tau$ is defined analogously to $K$ replacing $f$ by $f_\tau$.

The map $f_\tau$ depends on the twist parameter $\tau$ only up to multiples of $2\pi$.  However we will choose a conjugating homeomorphism $h_\tau$ which depends on the actual values of $\tau$.   The homeomorphism $h_\tau$ is defined as the homeomorphism satisfying the following properties:
\begin{enumerate}
 \item The map is of the form $h_\tau(e^{it},z) = (e^{it}, h_{\tau,t}(z))$.
 \item For each $t$ the map $h_{\tau,t}$ coincides with $z \mapsto e^{-i\rho(|z|)\tau(e^{it})}z$ on the pair of pants $P_t$, where $\rho$ is the bump function of Section \ref{auxiliarystep}.
 \item One has $f\circ h_\tau = h_\tau \circ f_\tau$ on all of $S^1 \times \C$.
\end{enumerate}

Pushing forward the metrics $g_s$ using $h_\tau$ one obtains an $f$ invariant family of metrics on $(S^1 \times \C) \setminus K$ and hence a hyperbolic metric on the Hirsch foliation $g_{\lambda,\tau}$.

\begin{remark} Let us emphasize that one should be careful with the interpretation of the twist function. When we glue pairs of pants via Smale's solenoid, there is naturally a twist in such a way that if $t>t'$ are close, $P_t$ and $P_{t'}$ are respectively glued to $P_{2t}$ and $P_{2t'}$ with twist parameters whose ratio is $e^{i(t-t')}$. This twist can't be undone for a topological reason: the quotient manifold $M$ is not Seifert. Hence our function $\tau$ has to be seen as an additional twist by that imposed by the dynamics of the foliation. Hence, even if we put a null twist function with a constant length function, the leaves won't be pairwise isometric (they are not even pairwise diffeomorphic).

\end{remark}

\subsection{Non-equivalence of model metrics}

We now show that construction we have given yields infinitely many non-equivalent hyperbolic metrics. The principal difficulty is of course to prove that two metrics corresponding to the same length functions, and to twist functions which differ from a multiple of $2\pi$ are not equivalent. Here this is more difficult than in the compact case where in order to prove that a Dehn twist is not isotopic to the identity, one chooses a transverse curve to that around which one perform the Dehn twist. Then one uses homology theory to show that the action of the Dehn twist on this curve is non-trivial. In the case of noncompact leaves, where we don't have a priori a transverse closed geodesic, this argument does not work: here the idea is to use the action of the twisting on those leaves \emph{with non-trivial holonomy}, in order to reduce the problem in a compact region.

\begin{theorem}\label{fenchelnielseninjectivity}
If $g_{\lambda,\tau}$ belongs to the same Teichmüller equivalence class as $g_{\lambda',\tau'}$ then $\lambda = \lambda'$ and $\tau = \tau'$.
\end{theorem}
\begin{proof}
Let $g_t$ and $g_t'$ be the family of metrics on the sets $\C_t = (\lbrace e^{it}\rbrace \times \C) \setminus K$ obtained by lifting $g_{\lambda,\tau}$ and $g_{\lambda',\tau'}$ to $(S^1 \times \C)\setminus K$.

Suppose there is a leaf preserving identity isotopy taking $g_{\lambda,\tau}$ to $g_{\lambda',\tau'}$.  Then lifting it one obtains an identity isotopy $F:[0,1] \times (S^1\times \C) \setminus K \to (S^1\times \C) \setminus K$ which preserves each $\C_t$ and such that $F_1$ sends $g_t$ to $g_t'$.

Since the outer boundary $C_t$ of the pair of pants $P_t$ is the unique closed geodesic in its isotopy class on its leaf for both $g_t$ and $g_t'$, one must have $F_1(C_t) = C_t$.  This implies that the length of $C_t$ is the same for both $g_t$ and $g_t'$ and since this is valid for all $t$ one obtains $\lambda = \lambda'$.

To show that $\tau = \tau'$ we first observe that any hyperbolic pair of pants with geodesic boundary components can be split in a unique way into two congruent hyperbolic right-angled hexagons.  This gives us a canonically defined pair of points which split each geodesic boundary comoponent in two.  If $g_{\lambda,\tau}$ and $g_{\lambda,\tau'}$ are equivalent then the angle between the thus obtained pairs of points on the boundaries of two neighboring pairs of pants in the foliation must coincide for both metrics.  This yields that there must exist an integer $n$ such that $\tau(e^{it}) - \tau'(e^{it}) = 2\pi n$ for all $t$.

To conclude the proof we will show that $n = 0$.  To see this consider the leaf containing the pair of pants $P_1 = \lbrace 1\rbrace \times P$.  

Notice that $P_1$ projects to a torus minus a disk $S$ in the foliation and that both metrics coincide on $S$.  Since $f_\tau = f_{\tau + 2\pi n} = f_{\tau'}$ one obtains that the map $h = h_{\tau'}^{-1}\circ h_{\tau}$ commutes with $f$ and hence projects to a leaf preserving homeomorphism $H$ of the Hirsch foliation.  If $\tau \neq \tau'$ then $H$ is a Dehn twist of angle $2\pi n$ around a meridian of $S$.  On the other hand $H$ is an isometry between the hyperbolic metrics $g_{\lambda,\tau}$ and $g_{\lambda,\tau'}$ on $S$.   Because both metrics coincide one has that $H$ is an self-isometry of $S$ endowed with either one of them and hence, since $S$ is compact, $H$ has finite order. But it is impossible if $n\neq 0$: the action of $H$ on a closed geodesic in $S$ transverse to the meridian has infinite order. This shows that $n= 0$ as claimed.
\end{proof}

%% file: deforming.tex
In this section we will prove that any hyperbolic metric on the Hirsch foliation can be deformed via an identity isotopy (which preserves and is smooth on each leaf) to one of the model metrics constructed in the previous section.

The proof will be carried out in two cases.  First, we show how to deform metrics for which the meridians (i.e. the outer boundaries of the pairs of pants $P_t$) are already geodesics.   Second we show how to deform a general hyperbolic metric to one with geodesic meridians.

\subsection{Metrics with geodesic meridians}
\input{deforminggeodesicmeridians.tex}
\subsection{Deforming general hyperbolic metrics}
\input{deforminggeneralmetrics.tex}

%% file: deforminggeodesicmeridians.tex
Let $g_{\lambda,\tau}$ be the family of metrics depending on two continuous functions from $S^1$ to $\R^{\ast}_+$ and $\R$ respectively, constructed in the previous section.  Recall that the meridians of the Hirsch foliation are the projections of the curves $\lbrace e^{it}\rbrace \times S^1 \subset S^1 \times \C$.  Notice that for the metrics $g_{\lambda,\tau}$ all meridians are geodesics.   In this subsection we will consider only metrics for which this is the case, proving the following:
\begin{proposition}
\label{geodesicmeridians}
Let $g$ be a hyperbolic metric on the Hirsch foliation for which all meridians are geodesics.  Then $g$ is equivalent to $g_{\lambda,\tau}$ for a unique choice of $\lambda:S^1 \to \R^{\ast}_+$ and $\tau:S^1 \to \R$.
\end{proposition}

\begin{proof}
The uniqueness claim follows immediately from Theorem \ref{fenchelnielseninjectivity}.

\paragraph{Making the metrics admissible near the meridians.} Let $g_t$ be the lift of the given hyperbolic metric to $\lbrace e^{it}\rbrace \times \C$.  By hypothesis the circle $\lbrace e^{it}\rbrace \times S^1$ is a geodesic for $g_t$.  We will first show that $g$ can be deformed to a metric whose lift is standard around each of these geodesics.

For this purpose we identify each annulus $\A_t = \lbrace e^{it} \rbrace \times \lbrace 0.9 < |z| < 1.1\rbrace$ with $\A = \lbrace 0.9 \l |z| \l 1.1\rbrace$ by projection.  Using this identification, the Massage Lemma (i.e. Lemma \ref{massagelemma}) yields a continuous family of identity isotopies supported in each $\A_t$ and which deform each metric $g_t$ to a standard metric around the unit circle.   Since each annulus $\A_t$ projects to the Hirsch foliation diffeomorphically this yields a well defined, leaf-preserving, identity isotopy on the Hirsch foliation.  After applying this identity isotopy one obtains a metric on the Hirsch foliation whose lift is standard around each meridian circle $\lbrace e^{it} \rbrace \times S^1$.  Notice that (since the family of metrics $g_t$ is $f$-invariant) this implies that one may assume from now on that each metric $g_t$ is admissible on the pair of pants $P_t$.

\paragraph{Deformation in a fundamental domain.} Assuming that each metric $g_t$ is admissible on $P_t$ we will now show how to deform $g$ to a model metric.   There is only one possible definition of the parameter $\lambda(e^{it})$ for the model metric, namely the length of the outer boundary of $P_t$ for the metric $g_t$.

To identify the twist parameter we first recall that each pair of pants $P_t$ is identified with the standard pair of pants $P$ via projection onto the second coordinate composed with the rotation $z \mapsto e^{-it}z$.  And recall, once again, that $P_t = P_{t+2\pi}$ but the identifications with $P$ differ by a rotation $R$ of angle $\pi$.

Consider the homotopy $H:[0,1]\times H(P) \to H(P)$ given by Theorem \ref{theperfectsection} which was also used to construct the model metrics in the previous section.   Because of the symmetry property of $H$ the continuous curve of metrics $s \mapsto H(s,g_t), s \in [0,1]$ is well defined on $P_t$.

\paragraph{Looking for the twist function.} Let $f_{s,t}:P_t \to P_t$ be the unique diffeomorphism isotopic to the identity such that $f_{s,t*}g_t = H(s,g_t)$ (see Corollary \ref{isotopiesvscurvesofmetrics}).  This family of diffeomorphisms on $P_t$ cannot in general be extended $f$-invariantly and hence do not define a leaf-preserving isotopy of the Hirsch foliation.

However, notice that since both $g_t$ and $H(s,g_t)$ are boundary admissible $f_{s,t}$ coincides with a Euclidean rotation on a neighborhood of each boundary component of $P_t$.

Let $\tau_s(e^{it})$ be the difference between the angle of rotation of the outer boundary of $P_t$ under $f_{s,t}$ and that of its image under $f_{s,2t}$ (there is a unique continuous way to choose $\tau_s$ starting with $\tau_0 = 0$).   For each $s \in [0,1]$ the family of diffeomorphisms $f_{s,t}$ can be extended continuously and uniquely in a $f_{\tau_s}$-invariant manner, where $f_{\tau_s}$ is the twisted solenoid map as defined for the construction of model metrics given in section \ref{modelmetrics}.   This implies that the diffeomorphisms defined on each $P_t$ by $h_{\tau_s}\circ f_{s,t}$ (where $h_{\tau_s}$ is the homeomorphism which conjugates $f$ and $f_{\tau_s}$ used for the construction of model metrics in section \ref{modelmetrics}) can be extended to a diffeomorphism of $(S^1 \times \C) \setminus K$ in a unique $f$-invariant manner.  This defines a leaf preserving identity isotopy of the Hirsch foliation.

To conclude, notice that the metric obtained by pushing $g_t$ forward by the time $1$ of the isotopy we just defined is defined on $P_t$ has $h_{\tau_1(e^{it})*}f_{s,t*}g_t = h_{\tau_1(e^{it})*}H(1,g_t)$.  Hence it coincides with the model metric $g_{\lambda,\tau}$ for $\tau = \tau_1$ and $\lambda(e^{it})$ defined above.   This concludes the proof.
\end{proof}

%% file: deforminggeneralmetrics.tex
 The goal here is to prove the following proposition:

\begin{proposition}
\label{makemeridiansgeo}
Let $g'$ be a hyperbolic metric for the Hirsch foliation. Then there is an identity isotopy of the Hirsch foliation $H:M\times[0,1]\to M$ such that the meridians are geodesics of $g=H_{\ast} g'$.
\end{proposition}

Recall that the manifold $M$ is obtained by gluing via Smale's solenoid $f$ the boundary components of the fundamental domain $M_0=\overline{\fT}\moins f(\fT)$. As we have already seen, this gluing operation determines a torus in $M$ that we will denote $T$ through this section, which intersects the leaves according to the meridians. Note that the exceptional fiber of $M_0$ provides a circle $S$ transverse to every leaf and enables us to parametrize the torus by the meridians, i.e. to write:

$$T=\bigcup_{s\in S} \gamma_s,$$
where the $\gamma_s$ are the meridians of $\F$. For this reason, we call $T$ the \emph{meridian torus} of $\F$. We recall that inside the manifold $M$, it is canonical in the sense that it is the unique incompressible torus up to isotopy. However, constructing an isotopy from another incompressible torus to $T$ which is leaf preserving is much more delicate. We propose to show below how to treat this difficulty.

\subsubsection{The geodesic torus}
Until the end of the paper, we fix a hyperbolic metric $g'$ for the Hirsch foliation $(M,\F)$.

\paragraph{Tubular neighborhoods.} Fixing a smooth vector field $X$ transverse to $\F$ identifies the normal bundle $TM/T\F$ with the subbundle $\NN^{\F}$ generated by $X$. The projection along the flow lines of $X$ provides a smooth submersion $\Pi$ of a tubular neighborhood of uniform size $\delta$ of the zero section, denoted by $\NN^{\F}_{\delta}$.

By shrinking the size of the neighborhood, we may assume that $\Pi$ 
induces a local diffeomorphism on the restriction of the normal bundle to any leaf $L$ of $\F$. Hence $\hcF=\Pi^{-1}(\F)$ is a smooth local foliation of $\NN^{\F}_{\delta}$ which induces $\F$ in the zero section.

\paragraph{Reeb's stability and geodesic torus.}

Let $o\in S$. The surface $(L_o,g_{L_o}')$ is hyperbolic and does not have any cusp: as a consequence there exists a unique geodesic $\gamma_o'$ in the free homotopy class of $\gamma_o$.

Consider now a small collar neighborhood $U\dans L_o$ of $\gamma_o'$. The foliation $\hcF$ induces a foliation of the tubular neighborhood $\NN_{\delta}^{\F}(U)$. The generator of $\pi_1(U)$ is without holonomy since it is freely homotopic to a meridian. Hence, the \emph{Reeb stability theorem} states that in a small tubular neighborhood $\NN_U$ of $U$, the foliation $\F$ induces a trivial foliation.

Parametrize the foliation of $\NN_U$ by coordinates $A\times(-1,1)$, where $A$ is an annulus and $U$ is sent on to $A\times\{0\}$: the metric $g$ induces a family of hyperbolic metrics on the $A\times\{s\}$ which varies continuously with $s$. In $A\times\{0\}$, we have a copy of $\gamma_o'$ which is a geodesic, denoted by $c_0$. Since the geodesic in a given free homotopy class varies continuously with the hyperbolic metric (this is due to the persistence of periodic orbits of the geodesic flow on a hyperbolic surface, which are normally hyperbolic), there exists a small $\delta'$ such that for all $s\in[-\delta',\delta']$, there is a unique geodesic $c_s$ of $A\times\{s\}$ in the non-trivial free homotopy class and that $s\mapsto c_s$ is continuous in the smooth topology. This gives sense to the following property:

\begin{lemma}
\label{geodesictorus}
The map $s\mapsto\gamma_s'$ is continuous in the $C^{\infty}$-topology. In particular, $T'=\bigcup_{s\in S} \gamma_s'$ forms a topological torus called the geodesic torus.
\end{lemma}

Our goal is to show that this geodesic torus can be leaf-isotoped to the meridian torus. By a foliated version of the theorem of extension of isotopies, it will be possible to construct the identity isotopy of the Hirsch foliation $(M,\F)$ that we look for.

\begin{proposition}
\label{isotopictori}
Let $g'$ be a hyperbolic metric on the Hirsch foliation. Let $T'$ be the associated geodesic torus. Then there exists a leaf isotopy $(\phi_t)_{t\in[0,1]}$ from $T'$ to the meridian torus $T$.
\end{proposition}

Epstein's work \cite{Epstein} implies that two freely homotopic curves on a surface are always isotopic. However, this theorem will be of little help for our purpose, since we want to isotope curves all together, in a continuous way, without creating self-intersection of the torus. Our approach will be to use the properties of the curve shortening flow in order to construct the desired isotopy.

\paragraph{Dealing with close tori.} It will be convenient in what follows to consider on $(M,\F)$ the metric $g_{1,0}$ associated to a length function (resp. twist function) which is identically $1$ (resp. $0$). Remark in particular that for such a metric, all the meridians are geodesics. We first show that the conclusion of Proposition \ref{isotopictori} holds when the two tori are very close.

\begin{lemma}
\label{closecurves}
There exists $\eps>0$ such that if for every $s\in S$, $\gamma_s$ and $\gamma'_s$ are $\eps$-close in the $C^1$-topology, then the conclusion of Proposition \ref{isotopictori} holds.
\end{lemma}

\begin{proof}
Consider the annulus $A=S^1\times (-1,1)$ endowed, say, with the usual Euclidean metric. Then there exists $\eps_0>0$ such that every curve $\gamma$ which is $\eps_0$ close to $S^1\times\{0\}$ in the $C^1$ topology can be written as the graph of a function $f:S^1\to (-1,1)$. This is just because if $\eps_0$ is small enough, any such curve has to remain uniformly transverse to the normal direction.

Now, consider the `collar neighborhood' $\NN$ of the meridian torus: that is to say the union of the collar neighborhoods of the meridians (see Section \ref{sectionstandardness}). By compactness of $S$ and continuity of the collar function, there exists an embedding $\sigma$ of $A\times S$ whose image is precisely $\NN$, and which sends every slice $A\times\{s\}$ on the collar neighborhood of $\gamma_s$ with a uniform Lipschitz constant. Of course, we ask that for every $s\in S$, $\sigma(.,0,s)=\gamma_s$.

The above shows that there exists a uniform $\eps>0$ such that any curve $\eps$-close to $\gamma_s$ can be parametrized as $\gamma(p)=\sigma(p,f(p),s)$, where $f:S^1\to(-1,1)$ is a smooth function.

Now assume that the geodesic torus $T'$ is $\eps$-close to the meridian torus  $T$ in the sense that for every $s\in S$, the curve $\gamma'_s$ is $\eps$-close to the meridian $\gamma_s$ in the $C^1$-topology. Then there is  smooth function $f:S^1\times S\to (-1,1)$ such that for every $s\in S$, $\gamma_s'$ is parametrized as $\sigma(p,f(p,s),s)$.

Now the map $(p,s,t)\in S^1\times S\times[0,1]\mapsto\sigma(p,(1-t)f(s,p),s)$ provides the desired isotopy.
\end{proof}

\subsubsection{The curve shortening flow}

\paragraph{Definition.} In order to isotope the geodesic torus on the meridian torus, it will be convenient to let the geodesics of $g'$ evolve along \emph{curve shortening flow}.

More precisely, we are interested in the following problem. Let $L$ one of the leaf of $\F$ endowed with the metric induced by $g_{0,1}$. Consider a smooth family of immersions of the circle $C(.,t):S^1\to L$  which solves the following system of PDE:
\begin{equation}
\label{CSFpb}
\begin{cases}\partial_t C(p,t)&=k(p,t)N(p,t)\\
             C(.,0)&=\gamma_o'
             
\end{cases}
\end{equation}
where $k(p,t)$ is the curvature of $C(p,t)$ and $N(p,t)$ the unit normal vector.

\paragraph{Properties.}
We will use the following result, which is a general theorem about the short-time existence and uniqueness of solutions of the curve shortening problem. Let us emphasize the fact that continuity with respect to the parameters (in particular with respect to the initial curve and the metric) is fundamental in our proof. At the end of this section, we will discuss in details how to get the following theorem from the classical existence theorems (see \cite{AngenentAnnals,AngenentEdin} and also \cite{Mantegazza2011}).
\begin{theorem}
\label{continuityCSF}
Let $(L,g)$ be a Riemannian surface, and $\gamma_o':S^1\to L$ be a smooth immersion. Then there exists a unique smooth solution to Problem \eqref{CSFpb} in some positive interval time.

Moreover the solution $C$ depends continuously on the initial immersion as well as on the metric $g$.
\end{theorem}

Let us come back to our situation where $L$ is a leaf of the Hirsch foliation endowed with the metric induced by $g_{0,1}$. As a Riemannian surface, the leaf is \emph{convex at infinity} in the sense of Grayson \cite{Grayson1989}: we can apply its results to the curves $\gamma_s'$.

\begin{theorem}[Grayson]
\label{propertiesCSF}
Let $L$ be a leaf of the Hirsch foliation endowed with the metric induced by $g_{0,1}$. Let $\gamma'_o$ be a smooth simple closed curve which is freely homotopic to a meridian $\gamma_o$. Then:
\begin{enumerate}
\item the solution of Problem \eqref{CSFpb} exists for every $t\in[0,\infty)$;
\item the curves $C(t)=C(S^1,t)$ stay embedded and converge to $C(\infty)=\gamma_o$ in the smooth topology;
\item if $\gamma_{s}'$ is disjoint from $\gamma_o'$ and freely homotopic to another meridian $\gamma_{s}$, then $C'(t)$ and $C(t)$ stay disjoint for all time, where $C'$ is the solution of \eqref{CSFpb} with $C'(0)=\gamma'_{s}$.
\end{enumerate}
\end{theorem}

Combining these two results, i.e. the existence of solutions of the curvature shortening flow for all times, and the local continuity with respect to the parameters, we obtain the following consequence.

\begin{proposition}
\label{timeT} 
Let $L$ be a leaf of the Hirsch foliation endowed with the metric induced by $g_{0,1}$. Let $\gamma'_o$ be a smooth simple closed curve which is freely homotopic to a meridian $\gamma_o$. Let $C$ be the solution of Problem \eqref{CSFpb}.

Then for every $\tau>0$, the function $p\in S^1\mapsto C(p,\tau)$ varies continuously in the $C^{\infty}$-topology with respect to the metric and to the initial curve $\gamma_o'$.
\end{proposition}

\paragraph{Leaf isotopy of the geodesic torus.} We are now going to use the curve shortening flow in order to get the desired leaf isotopy of the two tori.

\begin{lemma}
\label{isotopeCSF}
Let $T'$ be the geodesic torus associated to some hyperbolic metric $g'$ on the Hirsch foliation $(M,\F)$.

Endow $(M,\F)$ with the model metric $g_{0,1}$. Let $\eps>0$ be the number given by Lemma \ref{closecurves}. Denote, for every $s\in S$, by $C_s$ the solution of \eqref{CSFpb} with $C_s(.,0)=\gamma_s'$. Then:
\begin{enumerate}
\item there exists a time $\tau>0$ such that for every $s\in S$, $C_s(.,\tau)$ is $\eps$-close to $\gamma_s$ in the $C^1$-topology;
\item the function $(p,s,t)\in S^1\times S\times[0,\tau]\simeq T'\times[0,\tau]\mapsto C_s(p,t)$ provides a leaf isotopy between $T'$ and some torus $T^{\tau}$.
\end{enumerate}
\end{lemma}

Before we give the proof of this lemma, let us note that, together with Lemma \ref{closecurves}, it implies Proposition \ref{isotopictori}. One just has to let run the curve shortening flow until all curves $\gamma'_s$ become $\eps$-close to the corresponding meridian $\gamma_s$ in the $C^1$-topology, and then we isotope the resulting torus on $T$ thanks to the isotopy provided by Lemma \ref{closecurves}.

\paragraph{Proof of Lemma \ref{isotopeCSF}.}

Consider the geodesic torus $T'=\bigcup_{s\in S}\gamma_s'$ associated to the hyperbolic metric $g'$, and consider for every $s\in S$, the solution $C_s(.,t)$ of Problem \eqref{CSFpb} with $C(.,0)=\gamma'_s$. We have fixed a smooth vector field $X$ transverse to $\F$.

Note first that by Grayson's theorem, all $C_s(.,t)$, $s\in S,t>0$, are disjoint embeddings of $S^1$. We want to prove the continuity of these embeddings with the transverse parameter.

For $s\in S$, consider the cylinder $K_s^0=\bigcup_{t\in[0,\infty)} C_s(t)\dans L_s$, as well as its $\eps_1$-neighborhood $K_s$, where $\eps_1$ is a small constant. Note in particular that $\gamma_s'$ and $\gamma_s$ have images included in $K_s$.

Using once again the Reeb stability theorem in a sufficiently small tubular neighborhood provides a neighborhood of $K_o$ of the form $N=\bigcup_{s\in I} K_s$, where $I\dans S$ is a small neighborhood of $o$ and $K_s\dans L_s$ is a cylinder whose fundamental group is generated by the meridian $\gamma_s'$. 

Shrinking $I$ if necessary, we can assume that the projection under the flow lines of $X$ induces a smooth map
$$\pi_{s,s'}:\overline{K_s^0}\to K_{s'}$$
which is a diffeomorphism on its image.

Given $s\in I$, consider on $\pi_{s,o}(K_s^0)\dans K_o$ the metric $(\pi_{s,o})_{\ast} g_{L_s}$. When $s$ converges to $o$ in $I$, the metric converges to $g_{L_o}$: this is another way to express the continuity of the metric of the leaves with respect to the transverse parameter. We also have, by Lemma \ref{geodesictorus}, that $\pi_{s,o}(\gamma'_s)$ converges to $\gamma'_o$ in the $C^{\infty}$-topology.

Remark that the family of curves $\pi_{s,o}(C_s(.,t))\dans L_o$ are solution of the curve shortening problem for the metric $(\pi_{s,o})_{\ast} g_{L_0}$, with initial condition $\pi_{s,o}(C_s(.,t))$.

We can now use Lemma \ref{continuityCSF}: the solution at time $\tau$ of the curve shortening problems is continuous with respect to the initial curve, as well as with the metric. This gives sense to the following sentence: `the time $t$ of the curve shortening problem with initial condition $\gamma_s$ varies continuously with respect to the transverse parameter $s$'. In particular, for every $\tau>0$, the function $(p,s,t)\in S^1\times S\times[0,\tau]\simeq T\times[0,\tau]\mapsto C_s(p,t)$ provides a leaf isotopy between $T'$ and some torus $T^{\tau}$.

Using the compactness of $S$, as well as item 2 of Grayson's theorem \ref{propertiesCSF}, we immediately get that for some uniform $\tau>0$, all embeddings $C_s(.,\tau)$ are $\eps$-close to the corresponding meridian $\gamma_s$ in the $C^1$-topology, concluding the proof of Lemma \ref{isotopeCSF} and, thus, that of Proposition \ref{isotopictori}.

\subsubsection{Foliated isotopy extension}

To conclude the proof of Proposition \ref{makemeridiansgeo} it is enough to show how to extend the leaf isotopy $\phi_t$ from $T'$ to $T$ to a leaf isotopy of $(M,\F)$. Denote $T^t=\phi_t(T')$ for $t\in[0,1]$ and $\gamma_s^t=\phi_t(\gamma_s')$, for $s\in S$.

\begin{lemma}
\label{extension}
There exists a leaf isotopy $\Phi_t$ of the Hirsch foliation $(M,\F)$ which coincides with $\phi_t$ on $T$.
\end{lemma}

\begin{proof}
It is the same thing to ask for a leaf isotopy, or a continuous vector field $X=(H,1)$ on $M\times[0,1]$ such that:
\begin{enumerate}
\item  $H$ is tangent to $\F$ (we call this vector field the horizontal part of $X$);
\item  $H$ is smooth in the leaves of $\F$;
\item  $H$ varies continuously in the $C^{\infty}$-topology with the transverse parameter.
\end{enumerate}

The isotopy $\phi_s$ allows us to construct such a vector field $\bar{X}$ in $\TT=\bigcup_{t\in[0,1]} T^t\times\{t\}$. An argument of compactness, provides a number $\eps>0$ such that for every $t\in[0,1]$ if $\gamma^t_s$ and $\gamma^t_{s'}$ belong to the same leaf, they are distant of at least $\eps$.

For $s\in S$, $t\in[0,1]$, denote by $\nu_s^t$ the $\eps$-neighborhood of $\gamma_s^t$ in the corresponding leaf. The union $\NN$ of the $\nu_s^t\times\{t\}$ is a neighborhood of $\TT$ inside $M\times[0,1]$.

Use a bump function in order to extend $\bar{X}$ to a vector field $X=(H,1)$ with support in $\NN$ with a horizontal part which satisfies the desired property.
\end{proof}

\subsubsection{Continuity of the solutions of the curve shortening flow with respect to the initial data}

The first step in proving Therorem \ref{continuityCSF} is to note that there is a geometric invariance of the solutions of the curve shortening problem under tangential reparametrization. This essentially means that any family of immersions of the circle satisfying $\langle \partial_tC(p,t)|N(p,t)\rangle_g=k(p,t)$ can be globally reparametrized to a solution of \eqref{CSFpb} without changing $\gamma_o'$: see for example Proposition 1.3.4 and Corollary 1.3.5 of \cite{Mantegazza2011}.

Using the exponential map, for example, it is possible to extend the immersion $\gamma_o':S^1\to L$ to a smooth immersion of the cylinder $\sigma: S^1\times(-1,1)\to L$, such that $\sigma_{|\{0\}\times S^1}=\gamma'_o$. Then, any immersion $C:S^1\to L$ close enough to $\gamma'_o$ can be parametrized as $C(p)=\sigma(p,u(p))$, where $u:S^1\to(-1,1)$ is a smooth function.

Hence, proving the existence of local solutions for the curve shortening problem reduces to looking for smooth functions $u:S^1\times[0,T]\to (-1,1)$ which satisfy:

\begin{equation}
\label{CSFpb2}
\begin{cases} \langle \partial_t\sigma(p,u(p,t))|N(p,t)\rangle_g=k(p,t)\\
             u(0,.)=0.
             
\end{cases}
\end{equation}

The next step in proving the existence theorem, is to note that Problem \eqref{CSFpb2} reduces to a quasilinear equation on the function $u$ which is of the form:

\begin{equation}
\label{CSFpb3}
\begin{cases} \partial_t u=F(p,u,\partial_p u,\partial^2_p u)=\alpha(p,u,\partial_p u)\partial_p^2 u+\beta(p,u,\partial_p u)\\
             u(0,.)=0,
             \end{cases}
\end{equation}
where $\alpha,\beta$ are smooth functions and $\alpha>0$, i.e.
$$\partial_{\xi} F(p,u,\zeta,\xi)>0.$$

We refer to Section 3 of \cite{AngenentAnnals} for the derivation of this equation (the same computation is also performed in Section 1.5 of \cite{Mantegazza2011}). One checks that the $F$ depends only on the $2$-jet of $g$, and that it varies continuously with $g$ in the $C^{\infty}$-topology.

Now the approach of \cite{AngenentEdin} shows the uniquness, the existence, and which is of interest for our purpose, the continuous dependence on the parameters. 

Let us sketch Angenent's approach. Consider the \emph{Hölder spaces} $E_0=h^{\alpha}(S^1)$ and $E_1=h^{2+\alpha}(S^1)$ where $h^{\beta}(S^1)$ is defined as the completion of $C^{\infty}(S^1)$ for the usual Hölder norm of exponent $\beta$. They are Banach spaces.

Considering the open set $O_1\dans E_1$ formed by the functions $u\in E_1$ of modulus $<1$, we have a quasilinear operator
$$\mathbf{F}:u\in O_1\mapsto F(p,u,\partial_p u,\partial_p^2 u)\in E_0.$$

Problem \eqref{CSFpb3} is then equivalent to the problem $\partial_t u=\mathbf{F}(u)$, $u(0)=0$, which is precisely the one studied in \cite{AngenentEdin}. It is proven that the Fréchet derivative $D\mathbf{F}(u):E_1\to E_0$ satisfies the hypothesis of Theorem 2.7 of  \cite{AngenentEdin}, which relies on a fixed point argument in order to prove the existence and uniqueness of the equation (inside $E_1$) in some positive interval time. Theorem 2.8 of \cite{AngenentEdin} then appeals to the implicit function theorem in order to prove the continuity with respect to the parameters. 

Finally, Angenent shows in (\cite{AngenentAnnals},p.460) how this construction immediatly implies the smoothing effect of the equation. This shows that the local solution of the quasilinear problem $F$ \eqref{CSFpb3} is, as $u$, of class $C^{\infty}$, and that the variation with the initial conditions is continuous in the smooth topology. This ends the proof of Theorem \ref{continuityCSF}.

%% file: contractibility.tex
In the preceding sections we have completed the proof of Theorem \ref{theorema}, establishing that the Teichmüller space of the Hirsch foliation is homeomorphic to the space of continuous closed curves on the plane.  Notice that in particular the Teichmüller space is contractible.

\subsection{Product structure of the space of hyperbolic metrics: proof of Theorem \ref{theoremb}}

We will now prove Theorem \ref{theoremb} which states that projection onto Teichmüller space gives the space of hyperbolic metrics on the Hirsch foliation the structure of a trivial $\diffid(M,\F)$ principal bundle.  In particular $H(M,\F)$ is homeomorphic to the product $\diffid(M,\F) \times T(M,\F)$.

\begin{proof}[Proof of Theorem \ref{theoremb}]
We have already constructed a global continuous section (i.e. the model metrics of Section \ref{sectionmodels}) of the projection from the space of hyperbolic metrics $H(M,\F)$ to the Teichmüller space $T(M,\F)$.  Hence, all that remains is to show that the action of the space of diffeomorphisms $\diffid(M,\F)$ is free and that the bundle is trivial.

To establish the first point (that the action is free) consider a diffeomorphism $f \in \diffid(M,\F)$ fixing a hyperbolic metric $g$ on the foliation.  We will show that $f$ must be the identity map.

To see this notice that, by definition, there is a leaf preserving isotopy $F:[0,1] \times M \to M$ with $F_0$ equal to the identity and $F_1 = f$.  Pick any leaf $L$ in the foliation and consider a locally isometric cover $\pi:\D \to L$ from the Poincaré disk (with the usual hyperbolic metric $4/(1-|z|^2)^2|\d z|^2$).  One may lift the isotopy $F$ (restricted to $L$) to the disk via $\pi$ obtaining an isotopy $\varphi:[0,1]\times \D \to \D$ starting at the identity.  Since $M$ is compact one obtains that the hyperbolic distance between $\varphi_0(z)$ and $\varphi_1(z)$ is uniformly bounded.  This implies that $\varphi_1$ extends continuously to the boundary of $\D$ as the identity.  However $\varphi_1$ is a lift of $f$ restricted to $L$ and hence is an isometry of the Poincaré metric on $\D$.  This implies that $f$ is the identity on the leaf $L$ and, by the same argument, on all leaves of the Hirsch foliation.

This concludes the proof that the projection from $H(M,\F)$ to $T(M,\F)$ gives the former the structure of a $\diffid(M,\F)$ principal fiber bundle.

Since the base of the fiber bundle $T(M,\F)$ is contractible one concludes that the bundle is trivial and in particular that the space of metrics $H(M,\F)$ is homeomorphic to the product $T(M,\F) \times \diffid(M,\F)$.
\end{proof}

\subsection{Contractibility : proof of Theorem \ref{theoremc}}

To conclude, we will show that the space of hyperbolic metrics is contractible which implies (combined with Theorems \ref{theorema} and \ref{theoremb}) that the identity component of the space of leaf preserving diffeomorphisms of the Hirsch foliation must be contractible as well, i.e. Theorem \ref{theoremc}.

\begin{proof}[Proof of Theorem \ref{theoremc}]
In view of theorems \ref{theorema} and \ref{theoremb} it suffices to show that $H(M,\F)$ is contractible to conclude that $\diffid(M,\F)$ must be as well.

To see this fix a reference hyperbolic metric $g$ on the Hirsch foliation and an orientation (constant in any local foliated chart) for each leaf.  Following Sullivan (see \cite[Section 5]{Sullivan}) we consider the unit tangent bundle $\TT$ of the foliation with respect to $g$ and the function $\pi:\TT \times \D \to M$ such that for each $v \in \TT$ the map $\pi_v$ is the unique locally isometric orientation preserving cover from $\D$ (with the Poincaré metric) to the leaf containing $v$ which sends the unit tangent vector at $0$ pointing towards the positive real axis to $v$.

Hence $\TT$ has been identified with the space of locally isometric orientation preserving covers from $\D$ into $(M,\F)$.  As noted by Sullivan, under this identification, there is a natural action of the group $\Isom^+(\D)$ of orientation preserving isometries of $\D$ on $\TT \times \D$ where the action of an isometry $\gamma$ on a pair $(v,z)$ is defined as $(v', \gamma(z))$ where $v'$ satisfies $\pi_{v'} = \pi_v \circ \gamma^{-1}$.

Given any hyperbolic metric $g'$ on the Hirsch foliation and a vector $v \in \TT$ one may lift $g'$ to $\D$ using the map $\pi_v$ to obtain a hyperbolic metric whose Beltrami coefficient we denote by $\mu_v$.  This family of coefficients is invariant under the action of $\Isom^+(\D)$ on $\TT \times \D$.

We define a path of hyperbolic metrics $g_t: t \in [0,1]$ from $g$ to $g'$ by letting $g_t$ be the unique hyperbolic metric on the foliation whose family of Beltrami coefficients on $\TT \times \D$ (defined as in the previous paragraph) is $t\mu_v$.

To see that for each $t$ the Beltrami coefficients $t\mu_v$ really come from a conformal structure on the Hirsch foliation, notice that the pullback $f^*\mu$ of a Beltrami coefficient $\mu$ under a conformal map $f$ satisfies $f^*\mu(z) = \overline{f'(z)}\mu(f(z))/f'(z)$, and equality is preserved if one multiplies both sides by a factor $t$.

The fact that the metric $g_t$ exists for each $t$ is a consequence of Candel's work (in particular \cite[Corollary 4.2]{Candel}).  The fact that it varies continuously with respect to $t$ and the initial metric $g'$ follows from the continuity of the solutions of the Beltrami equation with respect to the coefficient (see \cite{ahlfors-bers1960}).
\end{proof}

%% file: main.bbl
\newcommand{\etalchar}[1]{$^{#1}$}